\documentclass{article}

\usepackage[english]{babel}

\usepackage[letterpaper,top=2cm,bottom=2cm,left=3cm,right=3cm,marginparwidth=1.75cm]{geometry}
\RequirePackage[dvips]{graphicx}
\usepackage{xcolor,float}
\usepackage{amsmath}
\usepackage{cite}
\usepackage{graphicx}
\usepackage[colorlinks=true, allcolors=blue]{hyperref}
\usepackage{amsmath}
\usepackage{graphicx}
\usepackage{amscd,amsfonts,amsmath,amssymb,amstext,eufrak,yfonts}
\usepackage{amstext,amsthm,mathbbol,mathrsfs,stmaryrd}
\usepackage[cmtip,arrow]{xy}
\usepackage{pb-diagram,pb-xy}
\usepackage{curves,graphicx,texdraw}
\usepackage{tikz}
\usepackage[colorlinks=true, allcolors=blue]{hyperref}
\usepackage{amscd,amsfonts,amsmath,amssymb,amstext,eufrak,yfonts}
\usepackage{amstext,amsthm,mathbbol,mathrsfs,stmaryrd}
\theoremstyle{definition}

\theoremstyle{plain}
\newtheorem{theorem}{\sc Theorem}[section]
\newtheorem{proposition}[theorem]{\sc Proposition}
\newtheorem{claim}{Claim}
\newtheorem{lemma}[theorem]{\sc Lemma}
\newtheorem{corollary}[theorem]{\sc Corollary}

\newtheorem{conjecture}[theorem]{\sc Conjecture}

\title{Structural Properties and Applications of the Augmented Sombor Index}
\author{Chunlei Xu$^{1,2}$, Kinkar Chandra Das$^{2,}$\footnote{Corresponding author.}, Jayanta Bera$^{2}$
}
\date{}

\begin{document}
\maketitle
\linespread{1.5}\selectfont
\begin{center}
   $^1$College of Mathematics Science, Inner Mongolia Minzu University, Tongliao 028000, China.

    $^{2}$Department of Mathematics, Sungkyunkwan
University, Suwon 16419, Republic of Korea.

E-mail: xuchunlei1981@sina.cn,~kinkardas2003@gmail.com,~jayantabera@g.skku.edu
\end{center}

\begin{abstract}

Topological indices are key quantitative descriptors in mathematical chemistry, unchanged under symmetry operations and retaining graph connectivity; they capture molecular structural features to provide insights into molecular stability and chemical properties, becoming indispensable in cheminformatics and theoretical chemistry. Among degree-based indices, the \textbf{Sombor index} is widely concerned for capturing structural information, and motivated by enhanced structural discrimination, the \textbf{augmented Sombor index} ($ASO$) is defined for a connected graph $\Omega$ with $|V(\Omega)|\geq 3$ as $$ASO(\Omega) = \sum_{v_iv_j\in E(\Omega)} \sqrt{\frac{{d_i^2 + d_j^2}}{d_i + d_j - 2}},$$ where $d_i$ and $d_j$ are the degrees of vertices $v_i$ and $v_j$, respectively. Within the scope of this study, we first establish several sharp bounds for the augmented Sombor index and characterize the extremal graphs attaining these bounds. In particular, we determine the minimum value of the $ASO$ index for unicyclic graphs with a prescribed girth and characterize all graphs achieving this minimum. We also identify the second maximum $ASO$ value among trees and characterize the corresponding extremal tree structures. Furthermore, the minimum and maximum values of the $ASO$ index for bipartite graphs and chemical graphs are obtained, together with a complete characterization of the associated extremal graphs. In addition, we characterize the chemical trees that maximize the $ASO$ index. The chemical applicability of the $ASO$ index is investigated through quantitative structure–property relationship (QSPR) analysis, supported by a comparative assessment of several variants of the Sombor index. Finally, we present concluding remarks and outline potential directions for future research on the augmented Sombor index of graphs.

\bigskip


\end{abstract}

\section{Introduction}
Mathematical chemistry is an interdisciplinary field that integrates mathematical methods with chemical principles to gain a deeper understanding of molecular structures, properties, and reactions. By representing molecules as graph structures, with vertices representing atoms and edges corresponding to chemical bonds, researchers can leverage graph-theoretical techniques to unravel the underlying properties of molecules. This approach enables the analysis of chemical behavior in a highly structured and computationally efficient manner.

A key research emphasis in mathematical chemistry lies in leveraging topological indices, which function as numerical descriptors extracted from molecular graphs. These indices encapsulate the connectivity patterns of atoms within a molecule and serve as quantitative tools for predicting various physicochemical properties. A typical topological index $TI(\Omega)$ is expressed as:
\begin{align}
    TI(\Omega)=\sum_{v_iv_j\in E(\Omega)}F(d_i,d_j)=\sum_{1\leq i\leq j\leq n-1}x_{ij}F(i,j),\nonumber
\end{align}
where
$\Omega=(V(\Omega),E(\Omega))$ is a simple graph representing a molecule, $d_i$ stand for the degree of vertex $v_i$, and $F$ constitutes a symmetric binary function dependent on the degrees of vertices $v_i$ and $v_j$ ($v_iv_j\in E(\Omega)$), $x_{ij}$ is the number of edge between $i$-degree vertex and $j$-degree vertex. These topological indices exhibit significant power owing to their invariance under symmetry operations that retain a graph’s connectivity structure, thereby rendering them dependable tools for investigating structure–property relationships.

Among the diverse range of topological invariants, degree-based indices—those that rely exclusively on vertex degrees—have proven to be particularly valuable due to their simplicity and computational efficiency. These indices often show strong correlations with various molecular properties, including boiling points, surface tension, and molar refraction, which makes them essential in the domain of Quantitative Structure–Property Relationships (QSPR). Through QSPR models, researchers can predict molecular properties based solely on structure, thereby reducing the need for costly experimental measurements. Among several types of indices investigated within this
 field, degree-based topological indices play a prominent role. Relevant examples can be found in \cite{ABR,AADA1,ARB,BDFG,D,G,XuHorolBuyan,YLG,MVT,CAM1,CAM6,lu22}.

Among degree-based topological indices, Sombor index (proposed in \cite{Gutman}) is a notable one, and it is defined for a graph $\Omega$ as
\begin{align}
    SO=SO(\Omega)=\sum_{v_iv_j\in E(\Omega
    )}\sqrt{d_i^2+d_j^2}.\nonumber
\end{align}
This topological index has attracted considerable attention due to its strong theoretical foundations and practical utility in molecular modeling. Numerous studies, covering both theoretical analyses and experimental investigations, have explored the $SO$ index, with experimental results highlighting solid relationships between the index and the measured physicochemical properties of chemical compounds \cite{Sultan33,SRKK,Das1,das22trees,HX,JSJ,LGYH,LSZRL,MEB,MM,MB,O,R,SC,XL,ZNKIA,nithyaa24,li22,gao24}.
 Other indices related to the $SO$ index include the elliptic Sombor index and the Euler Sombor index, whose definitions are as follows:
\begin{align}
    ESO=ESO(\Omega)=\sum_{v_iv_j\in E(\Omega)}(d_i+d_j)\sqrt{d_i^2+d_j^2},\nonumber
\end{align}
and
\begin{align}
   EU=EU(\Omega)=\sum_{v_iv_j\in E(\Omega)}\sqrt{d_i^2+d_j^2+d_id_j},\nonumber
\end{align}
respectively. Relevant findings concerning the elliptic Sombor index and the Euler Sombor index are available in \cite{ADF,KCDJJ,gutman24elliptic, M,raza25eulersombor,SSEL,ZTHAZ}.

Building upon $SO$, a novel degree-based topological descriptor named the augmented Sombor index ($ASO$) was introduced in \cite{DGA}. For a graph $\Omega$ with at least three vertices and no component isomorphic to $P_2$ (a path with two vertices), the $ASO$ index is defined as:
\begin{align}
    ASO=ASO(\Omega)&=\sum_{v_iv_j\in E(\Omega)}\,\sqrt{\frac{d_i^2+d_j^2}{d_i+d_j-2}}=\sum_{v_iv_j\in E(\Omega)}\,h(d_i,d_j),\nonumber
\end{align}
where
\begin{align}
    h(d_i,d_j)=\sqrt{\frac{d_i^2+d_j^2}{d_i+d_j-2}}.\nonumber
\end{align}
Very recently, the mathematical properties and chemical applications of the augmented Sombor index have been investigated in \cite{DASALI1,DGA}. This paper extends this line of research by studying the augmented Sombor index from both theoretical and chemical perspectives. From the theoretical viewpoint, we establish several sharp bounds for the augmented Sombor index over various graph classes and completely characterize the extremal graphs attaining these bounds. From the chemical perspective, we examine the applicability of the augmented Sombor index.
Throughout the paper, we use standard graph-theoretic notation. In particular, $P_n$, $C_n$, $S_n$, and $K_{p,q}$ $(p+q=n)$ denote the path, cycle, star, and complete bipartite graph of order $n$, respectively. Furthermore, $DS(p,q)$ denotes the double star of order $n=p+q+2$, where $p\ge q\ge 1$. Moreover, $\Omega - v_iv_j$ denotes the graph obtained from $\Omega$ by deleting the edge $v_iv_j$.

The remainder of the paper is organized as follows. In Section~\ref{section2}, we study extremal problems for the augmented Sombor index and identify the graphs attaining the corresponding bounds, including unicyclic graphs with fixed girth, trees, bipartite graphs, chemical graphs, and chemical trees. For each class, the extremal values of the $ASO$ index are determined and the associated extremal structures are described in detail. In Section~\ref{section4}, we investigate the chemical applicability of the augmented Sombor index via QSPR analysis and present a comparative study with several related Sombor-type indices. Finally, concluding remarks and possible directions for future research are given in Section ~\ref{section5}.

 \section{Bounds of the Augmented Sombor Index of graphs}\label{section2}
 In this section, we first present several lemmas, then establish the bounds for $ASO$, and proceed to characterize the graphs that attain these bounds.
\begin{lemma}\emph{\cite{DGA}}\label{lemma2.1}
     Let $\Omega$ be a graph of order $n\geq 5$. For any pendant edge
$v_iv_j\in E(\Omega)(d_i>d_j=1)$,
\begin{align}
\sqrt{5}\leq h(d_i,d_j)\leq \sqrt{\frac{1+(n-1)^2}{n-2}},\nonumber
 \end{align}
where the left equality holds if and only if $(d_i,d_j)=(2,1)$ or $(d_i,d_j)=(3,1)$, and the right equality holds if and only if $(d_i,d_j)=(n-1,1)$.
\end{lemma}

 \begin{lemma}\emph{\cite{DGA}}\label{lemma2.2}
Let $\Omega$ be a graph of order $n\geq 4$. For any non-pendant edge $v_iv_j\in E(\Omega)$,
\begin{align}
    2\leq h(d_i,d_j)\leq \frac{n-1}{\sqrt{n-2}},\nonumber
\end{align}
where the left equality holds if and only if $(d_i,d_j)=(2,2)$, and the right
 equality holds if and only if $(d_i,d_j)=(n-1,n-1)$.
 \end{lemma}

\begin{lemma}\label{kcd1}
Let $\Omega$ be a graph of order $n\geq 4$. For any non-pendant edge $v_iv_j\in E(\Omega)$ with $d_i\geq d_j\geq 2$ and $(d_i,d_j)\neq (2,2)$,
\begin{align}
h(d_i,d_j)\geq \sqrt{\frac{13}{3}}\label{kcd2}
\end{align}
with equality if and only if $(d_i,d_j)=(3,2)$.
 \end{lemma}

\begin{proof} Let $v_iv_j\in E(G)$ be any non-pendant edge with $d_i\geq d_j\geq 2$ and $(d_i,d_j)\neq (2,2)$. First we assume that $d_i\geq d_j=2$. If $d_i=3$, then we have
   $$h(d_i,d_j)=\sqrt{\frac{d^2_i+d^2_j}{d_i+d_j-2}}=\sqrt{\frac{13}{3}}.$$
The equality holds in (\ref{kcd2}). Otherwise, $d_i\geq 4$. Thus we have
    $$h(d_i,d_j)=h(d_i,2)=\sqrt{\frac{d^2_i+4}{d_i}}=\sqrt{d_i+\frac{4}{d_i}}\geq \sqrt{5}>\sqrt{\frac{13}{3}}$$
as $f(x)=x+\frac{4}{x}$ is an increasing function on $x\geq 4$. The inequality strictly holds in (\ref{kcd2}).

\vspace*{3mm}

Next we assume that $d_i\geq d_j\geq 3$. Thus we have
\begin{align*}
(d_i-3)^2+(d_j-3)^2\geq 0,~\mbox{ that is, }~d^2_i+d^2_j\geq 6\,(d_i+d_j-3)&=\frac{13}{3}\,(d_i+d_j-2)+\frac{5}{3}\,(d_i+d_j-5.6)\\
   &>\frac{13}{3}\,(d_i+d_j-2),
\end{align*}
which implies
 $$h(d_i,d_j)=\sqrt{\frac{d^2_i+d^2_j}{d_i+d_j-2}}>\sqrt{\frac{13}{3}}.$$
The inequality strictly holds in (\ref{kcd2}).
\end{proof}

\begin{lemma}{\rm \cite{DASALI1}} \label{k0}
Let $\Omega$ be a graph containing no isolated edge. If $v_iv_j\in E(\Omega)$ such that $\Omega-v_iv_j$ contains no isolated edge, then
$$ASO(\Omega) > ASO(\Omega-v_iv_j).$$
\end{lemma}
The following two results were established in \cite{DGA}.
\begin{lemma}{\rm \cite{DGA}} \label{k1}
Let $T$ be a tree of order $n\,(>3)$. Then
\begin{equation}\label{eq-trees}
2\,\sqrt{5}+2\,(n-3)\leq ASO(T)\leq (n-1)\,\sqrt{n+\frac{2}{n-2}}
\end{equation}
with equality on the left if and only if $T\cong P_n$, and equality on the right if and only if $T\cong S_n$.
\end{lemma}

\begin{lemma}\label{e444}{\rm \cite{DGA}}
Let $\Omega$ be a connected graph of order $n>2$. Then
$$
ASO(\Omega) \ge 2\,\sqrt{5} + 2\,(n-3)
$$
with equality if and only if $\Omega \cong P_n$.
\end{lemma}

We now investigate the minimum $ASO$-value for bipartite graphs. Let $\mathcal{B}_{p,q}$ represent the set of bipartite graphs whose bipartition sets have sizes $p$ and $q$, respectively. Let $\mathcal{B}_n$ denote the set of all bipartite graphs of order $n$. Since $P_n$ is a bipartite graph, the minimal value of $ASO$ within $\mathcal{B}_{p,q}$ can be determined.

\begin{proposition}
For any connected graph $\Omega\in\mathcal{B}_{n}$,
\begin{align*}
        ASO(\Omega)\geq 2\,\sqrt{5}+2\,(n-3)=ASO(P_n)
\end{align*}
with equality if and only if $\Omega\cong P_n$.
 \end{proposition}

 \begin{proof} Since connected graph $\Omega\in\mathcal{B}_n$, by Lemma~\ref{e444}, we have
$$
ASO(\Omega) \ge 2\,\sqrt{5} + 2\,(n-3).
$$
Moreover, equality holds if and only if $\Omega \cong P_n$.
\end{proof}

Based on the result from Lemma \ref{k0}, the following result can be easily obtained.
 \begin{theorem}\label{theorem5.2}
          For any graph $\Omega\in\mathcal{B}_{n}$,
     \begin{align}
         ASO(\Omega)\leq  \Big\lceil\frac{n}{2}\Big\rceil\Big\lfloor\frac{n}{2}\Big\rfloor\sqrt{\frac{\lceil\frac{n}{2}\rceil^2+\lfloor\frac{n}{2}\rfloor^2}{n-2}}\nonumber
     \end{align}
with equality if and only if $\Omega\cong K_{\lceil\frac{n}{2}\rceil,\lfloor\frac{n}{2}\rfloor}$.
\end{theorem}

\begin{proof}
  Now, by the  consequence of Lemma \ref{k0}, for any graph $\Omega\in\mathcal{B}_{p,q}$,
  \begin{align}
ASO(\Omega)\leq ASO(K_{p,q})=pq\sqrt{\frac{p^2+q^2}{p+q-2}}=q(n-q)\sqrt{\frac{q^2+(n-q)^2}{n-2}},\nonumber
  \end{align}
where $p+q=n$, $p\geq q\geq 1$. Then $1\leq q\leq \lfloor\frac{n}{2}\rfloor$. Now let
    $$f(x)=x(n-x)\sqrt{\frac{x^2+(n-x)^2}{n-2}},\,1\leq x\leq \Big\lfloor\frac{n}{2}\Big\rfloor.$$
Then
  \begin{align*}
      f^\prime(x)=\frac{(n-2\,x)(3\,x^2-3\,nx+n^2)}{\sqrt{n-2}\sqrt{2x^2-2nx+n^2}}\geq 0.
  \end{align*}
  It follows that $f(x)$ is increasing for $1\leq x\leq \lfloor n/2\rfloor$. Thus, $f(x)\leq f(\lfloor n/2\rfloor)$, that is, when $x=\lfloor\frac{n}{2}\rfloor$, $f(x)$ attains its maximal value. Therefore, for any $\Omega\in\mathcal{B}_{n}$,
  \begin{align}
      ASO(\Omega)\leq \Big\lceil\frac{n}{2}\Big\rceil\Big\lfloor\frac{n}{2}\Big\rfloor\sqrt{\frac{\lceil\frac{n}{2}\rceil^2+\lfloor\frac{n}{2}\rfloor^2}{n-2}}\nonumber
  \end{align}
with equality if and only if $\Omega\cong K_{\lceil\frac{n}{2}\rceil,\lfloor\frac{n}{2}\rfloor}$. We complete the proof of this theorem.
\end{proof}

The unicyclic graph $C_g\star P_{n-g}$ is defined as the graph resulting from attaching a path $P_{n-g}$ onto one vertex of cycle $C_g$. Let $\mathcal{U}_{n,g}$ denote the family of all unicyclic grapphs of order $n$ with girth $g$. The subsequent theorem establishes the minimum $ASO$-value for the unicyclic graphs with given girth.

\begin{theorem} Let $\Omega\in\mathcal{U}_{n,g}~ (n\geq 4,~ g\geq 3)$. If $g\leq n-2$, then
\begin{align}
ASO(\Omega)\geq \sqrt{5}+\sqrt{39}+2\,(n-4)\label{kd0}
\end{align}
with equality if and only if $\Omega\cong C_g\star P_{n-g}$.
\end{theorem}

\begin{proof} Let $p$ denote the number of pendant vertices of $\Omega$. Since $\Omega$ is a unicyclic graph with girth $g \le n-2$, it follows that $p \ge 1$. Let $v_{\Delta}$ be a vertex of maximum degree $\Delta$ in $\Omega$. Then $\Delta\geq 3$ as $g\leq n-2$. We now prove the following claim:

\vspace*{2mm}

\begin{claim}\label{c1} $x_{22}\leq n-4$.
\end{claim}

\vspace*{2mm}

\noindent
{\bf Proof of Claim \ref{c1}.} If $\Delta\geq 4$, then
     $$\sum\limits_{1\leq j\leq i\leq \Delta,\atop i\geq 4}\,x_{ij}\geq 4.$$
Since $\Omega$ is a unicyclic graph, $\sum\limits_{1\leq j\leq i\leq \Delta}\,x_{ij}=n$. Thus we obtain
$$x_{22}\leq x_{31}+x_{32}+x_{33}+x_{21}+x_{22}=\sum\limits_{1\leq j\leq i\leq \Delta}\,x_{ij}-\sum\limits_{1\leq j\leq i\leq \Delta,\atop i\geq 4}\,x_{ij}\leq  n-4.$$
Otherwise, $\Delta=3$. Without loss of generality, we can assume that $v_{\Delta}\in V(C_g)$. Since $g\leq n-2$, then there exists a vertex $v_j$ in $\Omega$ with degree $d_j=1$ such that $v_sv_j\in E(\Omega)$. First we assume that $v_{\Delta}v_j\notin E(\Omega)$. Thus we obtain
$$\sum\limits_{1\leq j\leq i=3}\,x_{ij}+x_{21}\geq 4.$$
Hence
$$x_{22}\leq \sum\limits_{1\leq j\leq i\leq \Delta}\,x_{ij}-\sum\limits_{1\leq j\leq i=3}\,x_{ij}-x_{21}\leq  n-4.$$

\vspace*{3mm}

Next, we assume that $v_{\Delta}v_j\in E(\Omega)$. Since $g\leq n-2$, then exists a vertex $v_i\in V(C_g)$ of degree $3$ such that $v_i\neq v_{\Delta}$. Then one can easily see that
 $$\sum\limits_{1\leq j\leq i=3}\,x_{ij}\geq 5~\mbox{ and hence }~x_{22}\leq  \sum\limits_{1\leq j\leq i\leq \Delta}\,x_{ij}-\sum\limits_{1\leq j\leq i=3}\,x_{ij}\leq n-5.$$
This proves the {\bf Claim \ref{c1}}.

\vspace*{3mm}

In accordance with the definition of the augmented Sombor index with Lemmas \ref{lemma2.1} and \ref{kcd1}, we obtain
\begin{align}
ASO(\Omega)=\sum_{\substack{1\leq j\leq i\leq \Delta, \\i+j\neq 2}}x_{ij}h(i,j)
&=x_{22}h(2,2)+\sum\limits_{2\leq i\leq \Delta}\,x_{i1}\,h(i,1)+\sum\limits_{2\leq j\leq i\leq \Delta,\atop i>2}\,x_{ij}\,h(i,j)\nonumber
\end{align}
\begin{align}
&\geq 2\,x_{22}+\sqrt{5}\,\sum\limits_{2\leq i\leq \Delta}\,x_{i1}+\sqrt{\frac{13}{3}}\,\sum\limits_{2\leq j\leq i\leq \Delta,\atop i>2}\,x_{ij}\label{kd1}\\[2mm]
&= 2\,x_{22}+p\,\sqrt{5}+\sqrt{\frac{13}{3}}\,(n-p-x_{22})\nonumber\\[2mm]
&=p\,\left(\sqrt{5}-\sqrt{\frac{13}{3}}\right)-x_{22}\,\left(\sqrt{\frac{13}{3}}-2\right)+\sqrt{\frac{13}{3}}\,n\nonumber\\[2mm]
&\geq \sqrt{5}-\sqrt{\frac{13}{3}}-\left(\sqrt{\frac{13}{3}}-2\right)\,(n-4)+\sqrt{\frac{13}{3}}\,n\label{kd2}\\[2mm]
&=\sqrt{5}+2\,(n-4)+\sqrt{39}.\nonumber
\end{align}
The first part of the proof is done.

\vspace*{3mm}

Suppose that equality holds in (\ref{kd0}). Then all the inequalities in the above must be equalities. From equality in (\ref{kd1}), we have $h(i,1)=\sqrt{5}$ for all $2\leq i\leq n-1$, and $h(i,j)=\sqrt{\frac{13}{3}}$ for $2\leq j\leq i\leq n-1$, $i>2$. From equality in (\ref{kd2}), we obtain $p=1$ and $x_{22}=n-4$. Since $\Omega$ is a unicyclic graph of order $n$ with girth $g\,(\leq n-2)$, there are exactly one pendant edge $v_iv_j\in E(G)$ with $(d_i,d_j)=(2,1)$, $n-4$ edges whose end vertices are of degree $2$, and three non-pendant edges $v_iv_j\in E(G)$ with $(d_i,d_j)=(3,2)$. Hence $\Omega\cong C_g\star P_{n-g}$.

\vspace*{3mm}

Conversely, let $\Omega\cong C_g\star P_{n-g}$. Then one can easily see that $x_{21}=1$, $x_{22}=n-4$ and $x_{32}=3$. Thus we obtain $ASO(\Omega)=\sqrt{5}+\sqrt{39}+2\,(n-4)$.
 \end{proof}

 \begin{lemma}\label{lemma2.3}
If $1\leq q\leq p\leq n-3$, then
     \begin{align}
    &p\sqrt{\frac{(p+1)^2+1}{p}}+q\sqrt{\frac{(q+1)^2+1}{q}}+\sqrt{\frac{(p+1)^2+(q+1)^2}{n-2}}\nonumber\\[2mm]
    &~~~~~~~~~~~~~~< (p+1)\sqrt{\frac{(p+2)^2+1}{p+1}}+(q-1)\sqrt{\frac{q^2+1}{q-1}}+\sqrt{\frac{(p+2)^2+q^2}{n-2}}.\label{inequality2}
     \end{align}
 \end{lemma}

 \begin{proof}
 Let
 \begin{align}
     f(x)&=(x+1)\sqrt{\frac{(x+2)^2+1}{x+1}}-x\sqrt{\frac{(x+1)^2+1}{x}}\nonumber\\[2mm]
     &=\sqrt{(x+1)(x+2)^2+(x+1)}-\sqrt{x(x+1)^2+x},~x\geq 1. \nonumber
 \end{align}
 It is easy to see that
 \begin{align}
    & 9\,x^6+63\,x^5+191\,x^4+293\,x^3+218\,x^2+46\,x-20>0~~~\textup{when}~ x\geq 1,\nonumber\\[1mm]
  \textup{that is,}~~~&   (3\,x^2+10\,x+9)^{2}x(x^2+2\,x+2)>(3\,x^2+4\,x+2)^{2}(x+1)(x^2+4\,x+5),\nonumber\\[1mm]
      \textup{that is,}~~~&(3\,x^2+10\,x+9)\sqrt{x(x^2+2\,x+2)}>(3\,x^2+4\,x+2)\sqrt{(x+1)(x^2+4\,x+5)},\nonumber\\[2mm]
       \textup{that is,}~~~&\frac{3\,x^2+10\,x+9}{2\sqrt{(x+1)(x^2+4\,x+5)}}>\frac{3\,x^2+4\,x+2}{2\sqrt{x(x^2+2\,x+2)}},\nonumber
       \end{align}
       \begin{align}
     \textup{that is,}~~~& f^\prime(x)=\frac{3\,x^2+10\,x+9}{2\sqrt{(x+1)(x^2+4\,x+5)}}-\frac{3\,x^2+4\,x+2}{2\sqrt{x(x^2+2\,x+2)}}>0.\nonumber
     \end{align}

\vspace*{2mm}

It is deduced that $f(x)$ is strictly increasing. Since
$p\geq q$, we obtain
 $$\sqrt{(p+2)^2+q^2}>\sqrt{(p+1)^2+(q+1)^2}.$$
Using the above results with $p\geq q$, we obtain
     \begin{align}
         &~~~~(p+1)\sqrt{\frac{(p+2)^2+1}{p+1}}+(q-1)\sqrt{\frac{q^2+1}{q-1}}+\sqrt{\frac{(p+2)^2+q^2}{n-2}}\nonumber\\[2mm]
         &~~~~~~~~~~~~- \left(p\sqrt{\frac{(p+1)^2+1}{p}}+q\sqrt{\frac{(q+1)^2+1}{q}}+\sqrt{\frac{(p+1)^2+(q+1)^2}{n-2}}\right)\nonumber\\[2mm]
         &>\left((p+1)\sqrt{\frac{(p+2)^2+1}{p+1}}-p\sqrt{\frac{(p+1)^2+1}{p}}\right)-\left(q\sqrt{\frac{(q+1)^2+1}{q}}-(q-1)\sqrt{\frac{q^2+1}{q-1}}\right)\nonumber\\[2mm]
         &=f(p)-f(q-1)>0.\nonumber
     \end{align}
     Therefore, inequality (\ref{inequality2}) is satisfied.
 \end{proof}

Building upon the result established in Lemma \ref{lemma2.3}, we are now able to determine the extremal values of $ASO$ index for double stars.
\begin{proposition}\label{s1}
Let $DS(p,q)$ be a double star of order $n\,(p+q=n-2,\,p\geq q\geq 1)$. Then
     \begin{align}
        &\Big\lceil\frac{n-2}{2}\Big\rceil\sqrt{\frac{\left(\lceil\frac{n-2}{2}\rceil+1\right)^2+1}{\lceil\frac{n-2}{2}\rceil}}+ \Big\lfloor\frac{n-2}{2}\Big\rfloor\sqrt{\frac{\left(\lfloor\frac{n-2}{2}\rfloor+1\right)^2+1}{\lfloor\frac{n-2}{2}\rfloor}}+\sqrt{\frac{\left(\lceil\frac{n-2}{2}\rceil+1\right)^2+\left(\lfloor\frac{n-2}{2}\rfloor+1\right)^2}{n-2}}\nonumber\\[3mm]
        &\leq ASO(DS(p,q))\leq (n-3)\sqrt{\frac{(n-2)^2+1}{n-3}}+\sqrt{\frac{(n-2)^2+4}{n-2}}+\sqrt{5},\nonumber
     \end{align}
    herein, equality is attained on the left side if and only if $T\cong DS(\lceil\frac{n-2}{2}\rceil,\lfloor\frac{n-2}{2}\rfloor)$, while equality is attained on the right side if and only if $T\cong DS(n-3,1)$.
 \end{proposition}

\begin{lemma}{\rm \cite{DGA}}\label{1ww1}
Let $\Omega$ be a graph of order $n\,(>8)$ with any edge $v_iv_j$. Then
$h(d_i,d_j)<h(n-2,n-2)<h(n-1,n-3)=h(n-2,1)<h(n-1,2)<h(n-1,n-2)<h(n-1,n-1)<h(n-1,1)$
for $(d_i,d_j)\notin \Big\{(n-1,1),\,(n-1,2),\,(n-2,1),\,(n-2,n-2),\,(n-1,n-3),\,(n-1,n-2),\,(n-1,n-1)\Big\}$.
\end{lemma}

\begin{lemma}\label{kh1} Let $T\,(\ncong S_n)$ be a tree of order $n>5$. Then
\begin{itemize}
    \item [\emph{(1)}] For any non-pendant edge $v_iv_j\in E(T)$ satisfying $2\leq d_j\leq d_i$,
$$\sqrt{\frac{d^2_i+d^2_j}{d_i+d_j-2}}\leq \sqrt{n-2+\frac{4}{n-2}}$$
with equality if and only if $d_i=n-2$, $d_j=2$.
\item [\emph{(2)}] For any pendant edge $v_iv_j\in E(T)$ satisfying $1=d_j<d_i$,
$$\sqrt{\frac{d^2_i+d^2_j}{d_i+d_j-2}}\leq \sqrt{n-1+\frac{2}{n-3}}$$
with equality if and only if $d_i=n-2$, $d_j=1$.
\end{itemize}
\end{lemma}

\begin{proof}
Since $T\ncong S_n$, it follows $\Delta(T)\leq n-2$. \\
$(1)$ Let $v_iv_j\in E(T)$ be an 
edge with $2\leq d_j\leq d_i\leq n-2$. Now,
\begin{align}
\frac{d^2_i+d^2_j}{d_i+d_j-2}&=d_i+2-\frac{d_j\,(d_i-d_j+2)-4}{d_i+d_j-2}.\label{1gg1}
\end{align}

\vspace*{2mm}

\noindent
${\bf Case\,1.}$ $d_i\leq n-4$. From (\ref{1gg1}), we obtain
\begin{align}
\frac{d^2_i+d^2_j}{d_i+d_j-2}&\leq n-2-\frac{d_j\,(d_i-d_j+2)-4}{d_i+d_j-2}.\label{1gg2}
\end{align}
Since $d_i\geq d_j\geq 2$, we have $d_j\,(d_i-d_j+2)\geq 4$, and hence
\begin{align}
    \frac{d_j\,(d_i-d_j+2)-4}{d_i+d_j-2}\geq 0.\label{inequality44}
\end{align}
Using inequality (\ref{inequality44}) in (\ref{1gg2}), we obtain
 $$\sqrt{\frac{d^2_i+d^2_j}{d_i+d_j-2}}\leq \sqrt{n-2}<\sqrt{n-2+\frac{4}{n-2}},$$
as required.

\vspace*{2mm}

\noindent
${\bf Case\,2.}$ $d_i=n-3$. Since $T$ is a tree, it follows that $2\leq d_j\leq 3$. 
 Thus, 
\begin{align}
\frac{d^2_i+d^2_j}{d_i+d_j-2}=\frac{(n-3)^2+d^2_j}{n+d_j-5}&\leq \max\left\{\frac{(n-3)^2+4}{n-3},\,\frac{(n-3)^2+9}{n-2}\right\}.\label{1gg3}
\end{align}
Since $n\geq 6$,
$$\frac{(n-3)^2+9}{n-2}<n-2+\frac{4}{n-2}>\frac{(n-3)^2+4}{n-3}.$$
Using the above result in (\ref{1gg3}), we obtain
 $$\sqrt{\frac{d^2_i+d^2_j}{d_i+d_j-2}}<\sqrt{n-2+\frac{4}{n-2}},$$
as required.

\vspace*{2mm}

\noindent
${\bf Case\,3.}$ $d_i=n-2$. Since $T$ is a tree, we have $d_i+d_j\leq n$. If $d_j=2$, then
  $$\sqrt{\frac{d^2_i+d^2_j}{d_i+d_j-2}}=\sqrt{n-2+\frac{4}{n-2}}$$
and hence the equality holds. Otherwise, $d_j\geq 3$. Thus, we have $d_i+d_j\geq n+1$, a contradiction.

\vspace*{3mm}

\noindent
$(2)$ For any edge $v_iv_j\in E(T)$ with $1=d_j<d_i\leq n-2$,
$$\sqrt{\frac{d^2_i+d^2_j}{d_i+d_j-2}}=\sqrt{\frac{d^2_i+1}{d_i-1}}=\sqrt{d_i+1+\frac{2}{d_i-1}}\leq \sqrt{n-1+\frac{2}{n-3}}$$
with equality if and only if $d_i=n-2$, $d_j=1$.

This finalizes the proof for the aforementioned result.
\end{proof}
\begin{corollary}\label{1kh1} Let $T\,(\ncong S_n)$ be a tree of order $n>5$. Then
for any edge $v_iv_j\in E(T)$ with $d_j\leq d_i$,
$$\sqrt{\frac{d^2_i+d^2_j}{d_i+d_j-2}}\leq h(n-2,1)=\sqrt{n-1+\frac{2}{n-3}}$$
with equality if and only if $d_i=n-2$, $d_j=1$.
\end{corollary}

\begin{theorem}\label{theorem2.14}
Let $T\,(\ncong S_n)$ be a tree of order $n$. Then
\begin{align}
ASO(T)\leq (n-3)\sqrt{\frac{(n-2)^2+1}{n-3}}+\sqrt{\frac{(n-2)^2+4}{n-2}}+\sqrt{5}\label{t1}
\end{align}
with equality if and only if $T\cong DS(n-3,1)$.
\end{theorem}

\begin{proof}
For $4\leq n\leq 6$, utilizing the Sage mathematical software \cite{ST}, the validity of the result (\ref{t1}) can be readily verified. Otherwise, $n\geq 7$. Since $T\ncong S_n$, it follows $\Delta\leq n-2$. Let $d$ be the diameter of tree $T$. Then $d\geq 3$.

\vspace*{1mm}

First we assume that $d=3$. Then $T\cong DS(p,q)$, where $p+q=n-2,\,p\geq q\geq 1$. By Proposition \ref{s1}, the inequality presented in (\ref{t1}) is established. Furthermore, equality in (\ref{t1}) is attained if and only if $T\cong DS(n-3,1)$.

\vspace*{1mm}

Next we assume that $d\geq 4$. Let $P_{d+1}:\,v_1v_2\ldots v_dv_{d+1}$ denote  a diametral path in tree $T$. By virtue of Lemma \ref{kh1}, for any non-pendant edge $v_iv_j\in E(T)$ satisfying $2\leq d_j\leq d_i\leq n-2$, we obtain
\begin{align}
\sqrt{\frac{d^2_i+d^2_j}{d_i+d_j-2}}\leq h(n-2,2)=\sqrt{n-2+\frac{4}{n-2}}\label{1mm11}
\end{align}
with equality if and only if $d_i=n-2$, $d_j=2$. For $v_2v_3\in E(T)$, we have $2\leq d_2\leq n-3$ and $d_3\geq 2$. Using (\ref{1mm11}), we obtain
\begin{align}
h(d_2,d_3)=\sqrt{\frac{d^2_2+d^2_3}{d_2+d_3-2}}<h(n-2,2)=\sqrt{n-2+\frac{4}{n-2}}.\label{mm11}
\end{align}

\vspace*{3mm}

Without loss of generality, we can assume that $d_2\geq d_d$. Let $d_2=a+1$ and $d_d=b+1$, where $a\geq b$. It is straightforward to observe $a+b\leq n-3$ and $1\leq b\leq \frac{n-3}{2}$. We now analyze the following cases:

\vspace*{3mm}

\noindent
${\bf Case\,1.}$ $b=1$. In this scenario, for $v_dv_{d+1}\in E(T)$, $h(2,1)=\sqrt{5}$. Let $S_1=\{v_2v_3,\,v_dv_{d+1}\}$. Using Corollary \ref{1kh1} with (\ref{mm11}), we obtain
\begin{align*}
ASO(T)&=\sum\limits_{v_iv_j\in E(T)}\,h(d_i.d_j)\nonumber\\[2mm]
&=h(2,1)+h(d_2,d_3)+\sum\limits_{v_iv_j\in E(T)\backslash S_1}\,\sqrt{\frac{d^2_i+d^2_j}{d_i+d_j-2}}\\[2mm]
&<\sqrt{5}+h(n-2,2)+(n-3)\,h(n-2,1)\\[2mm]
&\leq (n-3)\sqrt{\frac{(n-2)^2+1}{n-3}}+\sqrt{\frac{(n-2)^2+4}{n-2}}+\sqrt{5}.
\end{align*}
The inequality (\ref{t1}) strictly holds.

\vspace*{2mm}

\noindent
${\bf Case\,2.}$ $2\leq b\leq \frac{n-3}{2}$. In this scenario, 
\begin{align}
\sum\limits_{v_j:v_2v_j\in E(T),\atop j\neq 3}\,\sqrt{\frac{d^2_2+d^2_j}{d_2+d_j-2}}&=a\,\sqrt{\frac{(a+1)^2+1}{a}}=a\,\sqrt{a+2+\frac{2}{a}}\label{inequality99}
\end{align}
and
\begin{align}
\sum\limits_{v_j:v_dv_j\in E(T),\atop j\neq d-1}\,\sqrt{\frac{d^2_d+d^2_j}{d_d+d_j-2}}&=b\,\sqrt{\frac{(b+1)^2+1}{b}}=b\,\sqrt{b+2+\frac{2}{b}}.\label{inequality1010}
\end{align}
Let $a+b=k$. Then we have $4\leq k\leq n-3$. From the above inequalities (\ref{inequality99}) and (\ref{inequality1010}), we obtain
\begin{align}
&\sum\limits_{v_j:v_2v_j\in E(T),\atop j\neq 3}\,\sqrt{\frac{d^2_2+d^2_j}{d_2+d_j-2}}+\sum\limits_{v_j:v_dv_j\in E(T),\atop j\neq d-1}\,\sqrt{\frac{d^2_d+d^2_j}{d_d+d_j-2}}\nonumber\\[3mm]
=&~a\,\sqrt{a+2+\frac{2}{a}}+b\,\sqrt{b+2+\frac{2}{b}}\nonumber\\[3mm]
=&~ (k-b)\,\sqrt{k-b+2+\frac{2}{k-b}}+b\,\sqrt{b+2+\frac{2}{b}}.\label{1t1}
\end{align}

We now define a function
 $$f(x)=(k-x)\,\sqrt{k-x+2+\frac{2}{k-x}}+x\,\sqrt{x+2+\frac{2}{x}},~2\leq x\leq \frac{k}{2}.$$
For $2\leq x\leq \frac{k}{2}$, we have $k-2\,x\geq 0$. Therefore
\begin{align}
    (x+2)\,(k-x)=x\,(k-x+2)+2\,(k-2\,x)\geq x\,(k-x+2).\label{inequality9}
\end{align}
Moreover, we have $k-x\geq x\geq 2$. Thus we obtain
$$(k-2\,x)\,\left(1-\frac{2}{x\,(k-x)}\right)\geq 0,~\mbox{ that is, }~(k-2\,x)+2\,\left(\frac{1}{k-x}-\frac{1}{x}\right)\geq 0,$$
which means
\begin{align}
\sqrt{k-x+2+\frac{2}{k-x}}\geq \sqrt{x+2+\frac{2}{x}}.\label{1ws1}
\end{align}
Using the above inequalities (\ref{inequality9})  and (\ref{1ws1}), we obtain
\begin{align}
(x+2)\,(k-x)\,\sqrt{k-x+2+\frac{2}{k-x}}\geq x\,(k-x+2)\,\sqrt{x+2+\frac{2}{x}}.\nonumber
\end{align}
Using the above result, we obtain
\begin{align}
&~~~~\frac{k-x+2}{(k-x)\,\sqrt{k-x+2+\dfrac{2}{k-x}}}-\frac{x+2}{x\,\sqrt{x+2+\dfrac{2}{x}}}\nonumber\\[3mm]
&=\frac{x\,(k-x+2)\,\sqrt{x+2+\dfrac{2}{x}}-(x+2)\,(k-x)\,\sqrt{k-x+2+\dfrac{2}{k-x}}}{x\,(k-x)\,\sqrt{k-x+2+\dfrac{2}{k-x}}\,\sqrt{x+2+\dfrac{2}{x}}}~\leq~ 0.\label{1ws2}
\end{align}
Using (\ref{1ws1}) and (\ref{1ws2}), we obtain
\begin{align}
&-\frac{(k-x)\,\left(1-\displaystyle{\frac{2}{(k-x)^2}}\right)}{2\,\sqrt{k-x+2+\displaystyle{\frac{2}{k-x}}}}+\frac{x\,\left(1-\displaystyle{\frac{2}{x^2}}\right)}{2\,\sqrt{x+2+\displaystyle{\frac{2}{x}}}}\nonumber\\[2mm]
=&-\frac{k-x+2+\displaystyle{\frac{2}{k-x}}-2-\displaystyle{\frac{4}{k-x}}}{2\,\sqrt{k-x+2+\displaystyle{\frac{2}{k-x}}}}+\frac{x+2+\displaystyle{\frac{2}{x}}-2-\displaystyle{\frac{4}{x}}}{2\,\sqrt{x+2+\displaystyle{\frac{2}{x}}}}\nonumber\\[2mm]
=&-\frac{1}{2}\,\sqrt{k-x+2+\frac{2}{k-x}}+\frac{1+\displaystyle{\frac{2}{k-x}}}{\sqrt{k-x+2+\displaystyle{\frac{2}{k-x}}}}+\frac{1}{2}\,\sqrt{x+2+\frac{2}{x}}-\frac{1+\displaystyle{\frac{2}{x}}}{\sqrt{x+2+\displaystyle{\frac{2}{x}}}}\nonumber\\[2mm]
=&-\frac{1}{2}\,\sqrt{k-x+2+\frac{2}{k-x}}+\frac{1}{2}\,\sqrt{x+2+\frac{2}{x}}+\frac{k-x+2}{(k-x)\,\sqrt{k-x+2+\dfrac{2}{k-x}}}-\frac{x+2}{x\,\sqrt{x+2+\dfrac{2}{x}}}\nonumber\\[2mm]
\leq &-\frac{1}{2}\,\sqrt{k-x+2+\frac{2}{k-x}}+\frac{1}{2}\,\sqrt{x+2+\frac{2}{x}}~\leq~ 0.\label{1ws3}
\end{align}
Then we have
\begin{align*}
f'(x)&=-\sqrt{k-x+2+\frac{2}{k-x}}-\frac{(k-x)\,\left(1-\displaystyle{\frac{2}{(k-x)^2}}\right)}{2\,\sqrt{k-x+2+\displaystyle{\frac{2}{k-x}}}}+\sqrt{x+2+\frac{2}{x}}+\frac{x\,\left(1-\displaystyle{\frac{2}{x^2}}\right)}{2\,\sqrt{x+2+\displaystyle{\frac{2}{x}}}}\leq 0\nonumber
\end{align*}
by (\ref{1ws1}) and (\ref{1ws3}). Hence, from (\ref{1t1}) we obtain
\begin{align}
   \sum\limits_{v_j:v_2v_j\in E(T),\atop j\neq 3}\,\sqrt{\frac{d^2_2+d^2_j}{d_2+d_j-2}}+\sum\limits_{v_j:v_dv_j\in E(T),\atop j\neq d-1}\,\sqrt{\frac{d^2_d+d^2_j}{d_d+d_j-2}}\leq f(2)=(k-2)\,\sqrt{k+\frac{2}{k-2}}+2\,\sqrt{5}.\label{inequality17}
\end{align}

\vspace*{2mm}

Let $S_3=\{v_2v_j\in E(T):\,j\neq 3\}\cup \{v_2v_3\}\cup \{v_dv_j\in E(T):\,j\neq d-1\}$. Using the above inequality (\ref{inequality17}) with (\ref{mm11}) and Corollary \ref{1kh1}, we obtain
\begin{align*}
ASO(T)&=\sum_{v_iv_j\in E(T)}h(d_i,d_j)\nonumber\\[2mm]
&=\sum\limits_{v_j:v_2v_j\in E(T),\atop j\neq 3}\,\sqrt{\frac{d^2_2+d^2_j}{d_2+d_j-2}}+\sum\limits_{v_j:v_dv_j\in E(T),\atop j\neq d-1}\,\sqrt{\frac{d^2_d+d^2_j}{d_d+d_j-2}}+\sum\limits_{v_2v_3\in E(T)}\,\sqrt{\frac{d^2_2+d^2_3}{d_2+d_3-2}}\\[2mm]
&~~~~~~~~~~~~~~~~~~~~~~~~~~~+\sum\limits_{v_iv_j\in E(T)\backslash S_3}\,\sqrt{\frac{d^2_i+d^2_j}{d_i+d_j-2}}\\[2mm]
&< (k-2)\,\sqrt{k+\frac{2}{k-2}}+2\,\sqrt{5}+\sqrt{\frac{(n-2)^2+4}{n-2}}+(n-k-2)\,\sqrt{\frac{(n-2)^2+1}{n-3}}\\[2mm]
&<(n-3)\sqrt{\frac{(n-2)^2+1}{n-3}}+\sqrt{\frac{(n-2)^2+4}{n-2}}+\sqrt{5},
\end{align*}
as $4\leq k\leq n-3$ and
$$\sqrt{5}\leq \sqrt{k+\frac{2}{k-2}}<\sqrt{\frac{(n-2)^2+1}{n-3}}.$$
This completes the proof of the theorem.
\end{proof}

As a preliminary definition, we recall that a chemical graph is connected and simple such that no vertex has a degree greater than four. We further denote $\mathcal{CG}_n$  as the class of chemical graphs of order $n$ and $\mathcal{CT}_n$ as the class of chemical trees of order $n$.

\vspace*{1mm}

For any graph $\Omega$ with order $n$, let $n_i$ denote the number of vertices of degree $i$. Then the following system of equations holds:
   $$\left\{\begin{aligned}
  &~ n_1+n_2+n_3+n_4+\cdots+n_{n-1}=n,\nonumber\\
    &~2\,x_{11}+x_{12}+x_{13}+\cdots+x_{1,n-1}=n_1,\nonumber\\
    &~x_{12}+2\,x_{22}+x_{23}+\cdots+x_{2,n-1}=2\,n_2,\nonumber\\
    &~x_{13}+x_{23}+2\,x_{33}+\cdots+x_{3,n-1}=3\,n_3,\nonumber\\
   &~~~~~~~~\cdots\cdots\cdots\cdots\nonumber\\
    &~x_{1,n-1}+x_{2,n-1}+\cdots+2\,x_{n-1,n-1}=(n-1)\,n_{n-1}.\nonumber
\end{aligned}
\right.$$
From the above $n$ equations, we obtain
\begin{eqnarray}
n=\sum_{1\leq i\leq j\leq n-1}\left(\frac{1}{i}+\frac{1}{j}\right)x_{ij}.\nonumber
\end{eqnarray}

For every $T\in \mathcal{CT}_n$ of order $n\geq 3$, it is evident that $x_{11}(T)=0$. Write
$$A=\Big\{(i,j):1\leq i\leq j\leq 4\Big\},~A_1=A\setminus\Big\{(1,4),\,(4,4),\,(1,1)\Big\}.$$
Then for chemical trees, this leads to
\begin{eqnarray}
 \frac{5}{4}\,x_{14}+\frac{1}{2}\,x_{44}+\sum_{(i,j)\in A_1}\left(\frac{1}{i}+\frac{1}{j}\right)x_{ij}=n,~~ x_{14}+x_{44}+\sum_{(i,j)\in A_1}x_{ij}=n-1.\nonumber
\end{eqnarray}
Thus, we have the expressions for $x_{14}$ and $x_{44}$ as follows
\begin{eqnarray}
x_{14}=\frac{2\,n}{3}+\frac{2}{3}-\frac{4}{3}\sum_{(i,j)\in
A_1}\left(\frac{1}{i}+\frac{1}{j}-\frac{1}{2}\right)x_{ij},~x_{44}=\frac{n}{3}-\frac{5}{3}+\frac{4}{3}\sum_{(i,j)\in
A_1}\left(\frac{1}{i}+\frac{1}{j}-\frac{5}{4}\right)x_{ij}.\nonumber
\end{eqnarray}

We define
\begin{align}
\phi_1(i,j)=\frac{4\,\sqrt{3}}{3}\left(\frac{1}{i}+\frac{1}{j}-\frac{5}{4}\right)-\frac{\sqrt{51}}{3}\left(\frac{1}{i}+\frac{1}{j}-\frac{1}{2}\right)+\frac{3}{4}\sqrt{\frac{i^2+j^2}{i+j-2}}~\mbox{ for any }(i,j)\in A_1.\nonumber
\end{align}
We now introduce three sets as follows
\begin{align}
\Gamma_1&=\Big\{T\in\mathcal{CT}_n: n_2=n_3=0\Big\},\,\Gamma_2=\Big\{T\in\mathcal{CT}_n: n_2=1,\,n_3=0\Big\},\,\Gamma_3=\Big\{T\in\mathcal{CT}_n:n_2=0,\,n_3=1\Big\}.\nonumber
\end{align}
With these sets in place, we can now state the following theorem, which determines the maximum value of $ASO$ index for chemical trees.
\begin{theorem} \label{th3.1}
For every tree $T\in \mathcal{CT}_n~(n\geq 5)$,
 $$ASO(T)\leq \frac{2\,\sqrt{51}}{9}\,(n+1)+\frac{4\,\sqrt{3}}{9}\,(n-5)$$
 with equality if and only if $T\in\Gamma_1$.
\end{theorem}

\begin{proof} By virtue of the definition of $ASO$,
    \begin{align}
       & ASO(T)=\sum_{(i,j)\in A}x_{ij}h(i,j)=\frac{\sqrt{51}}{3}\,x_{14}+\frac{4\,\sqrt{3}}{3}\,x_{44}+\sum_{(i,j)\in A_1}x_{ij}h(i,j)\nonumber\\[2mm]
        &=\frac{\sqrt{51}}{3}\left[\frac{2\,n}{3}+\frac{2}{3}-\frac{4}{3}\sum_{(i,j)\in
A_1}\left(\frac{1}{i}+\frac{1}{j}-\frac{1}{2}\right)x_{ij}\right]+\frac{4\,\sqrt{3}}{3}\left[\frac{n}{3}-\frac{5}{3}+\frac{4}{3}\sum_{(i,j)\in A_1}\left(\frac{1}{i}+\frac{1}{j}-\frac{5}{4}\right)x_{ij}\right]\nonumber\\[2mm]
&~~~~~~~~~~~~~~~~~~~~~~~~~~~+\sum_{(i,j)\in A_1}x_{ij}\sqrt{\frac{i^2+j^2}{i+j-2}}\nonumber\\[2mm]
&=\frac{\sqrt{51}}{3}\cdot\frac{2\,n+2}{3}+\frac{4\,\sqrt{3}}{3}\cdot\frac{n-5}{3}+\frac{4}{3}\,\sum_{(i,j)\in A_1}\left[\frac{4\,\sqrt{3}}{3}\left(\frac{1}{i}+\frac{1}{j}-\frac{5}{4}\right)-\frac{\sqrt{51}}{3}\left(\frac{1}{i}+\frac{1}{j}-\frac{1}{2}\right)+\frac{3}{4}\sqrt{\frac{i^2+j^2}{i+j-2}}\right]x_{ij}\nonumber\\[2mm]
&=\frac{2\,\sqrt{51}}{9}\,(n+1)+\frac{4\,\sqrt{3}}{9}\,(n-5)+\frac{4}{3}\,\sum_{(i,j)\in A_1}\,\phi_1(i,j)x_{ij}\label{1kinkar1}\\[2mm]
&\leq \frac{2\,\sqrt{51}}{9}\,(n+1)+\frac{4\,\sqrt{3}}{9}\,(n-5)\nonumber
\end{align}
since for any $(i,j)\in A_1$,
$$\phi_1(i,j)=\frac{4\,\sqrt{3}}{3}\left(\frac{1}{i}+\frac{1}{j}-\frac{5}{4}\right)-\frac{\sqrt{51}}{3}\left(\frac{1}{i}+\frac{1}{j}-\frac{1}{2}\right)+\frac{3}{4}\sqrt{\frac{i^2+j^2}{i+j-2}}<0,$$
which can be seen in Table \ref{table1}.
\begin{table}[h!]
\begin{center}
\begin{tabular}{|c|c|}
\hline
$(i,j)$ & $\phi_1(i,j)$\\
\hline
$(1,2)$ & $\frac{\sqrt{3}}{3}-\frac{\sqrt{51}}{3}+\frac{3\,\sqrt{5}}{4}\approx -0.126$ \\
\hline
$(1,3)$ &  $\frac{\sqrt{3}}{9}-\frac{5\,\sqrt{51}}{18}+\frac{3\,\sqrt{5}}{4}\approx -0.114$  \\
\hline
$(2,2)$ & $-\frac{\sqrt{3}}{3}-\frac{\sqrt{51}}{6}+\frac{3}{2}\approx -0.268$\\
\hline
$(2,3)$ & $-\frac{5\,\sqrt{3}}{9}-\frac{\sqrt{51}}{9}+\frac{\sqrt{39}}{4}\approx -0.194$ \\
\hline
$(2,4)$ & $-\frac{2\,\sqrt{3}}{3}-\frac{\sqrt{51}}{12}+\frac{3\,\sqrt{5}}{4}\approx -0.073$ \\
\hline
$(3,3)$ & $\frac{7\,\sqrt{3}}{9}-\frac{\sqrt{51}}{18}+\frac{9\,\sqrt{2}}{8}\approx -0.153$ \\
\hline
$(3,4)$ & $ -\frac{8\,\sqrt{3}}{9}-\frac{\sqrt{51}}{36}+\frac{3\,\sqrt{5}}{4}\approx -0.061$\\
\hline
\end{tabular}
\caption{Approximate values of function $\phi_1(i,j)$ on $A_1$}\label{table1}
\end{center}
\end{table}
Then the equality holds if and only if $x_{ij}=0$ for all $(i, j)\in A_1$, that is, if and only if $x_{12}=x_{13}=x_{22}=x_{23}=x_{24}=x_{13}=x_{33}=x_{34}=0$, that is, if and only if $n_2=n_3=0$, that is, if and only if $T\in\Gamma_1$.
\end{proof}
\begin{corollary}
     Let $n$ be an integer with $n\geq 5
     $ and $n\equiv 2\pmod 3$. If $T$ has maximum $ASO$-value in $\mathcal{CT}_n$, then $T\in\Gamma_1$.
\end{corollary}

\begin{proof}
   Let $T$ be a tree in $\mathcal{CT}_n$ that attains the maximum $ASO$-value. Then by Theorem \ref{th3.1}, we have $ASO(T)=\frac{2\,\sqrt{51}}{9}\,(n+1)+\frac{4\,\sqrt{3}}{9}\,(n-5)$. Moreover, $x_{ij}=0$ for any $(i,j)\in A_1$, that is, $n_2=n_3=0$, meaning that  $T\in\Gamma_1$.
\end{proof}

The maximal chemical trees with respect to $ASO$ for $n=4,7,10$ are easily obtained by using Sage~\cite{ST}, as shown in Figure \ref{6}.
\begin{figure}[h]
\begin{center}
\begin{tikzpicture}
\clip(-1.5,-2.5) rectangle (16,1.8);
\draw [line width=0.75pt] (0,0)--(1,0);
 \draw [line width=0.75pt] (1,0)--(2,0.8);
 \draw [line width=0.75pt] (1,0)--(2,-0.8);

\draw [line width=0.75pt] (4,0)--(7,0);
 \draw [line width=0.75pt] (5,0)--(5,0.8);

 \draw [line width=0.75pt] (6,0)--(6,0.8);
 \draw [line width=0.75pt] (6,0)--(6,-0.8);

\draw [line width=0.75pt] (9,0)--(13,0);
\draw [line width=0.75pt] (10,0)--(10,0.8);
\draw [line width=0.75pt] (10,0)--(10,-0.8);
\draw [line width=0.75pt] (11,0)--(11,0.8);
\draw [line width=0.75pt] (12,0)--(12,0.8);
\draw [line width=0.75pt] (12,0)--(12,-0.8);


 \draw (0.75,-1.5) node [anchor=north west]{$H_1$};
\draw (5.25,-1.5) node [anchor=north west]{$H_2$};
\draw (10.75,-1.5) node [anchor=north west]{$H_3$};

\begin{scriptsize}

\draw [fill=black] (0,0) circle (1.5pt);
 \draw [fill=black] (2,0.8) circle (1.5pt);
\draw [fill=black] (1,0) circle (1.5pt);
\draw [fill=black] (2,-0.8) circle (1.5pt);

\draw [fill=black] (4,0) circle (1.5pt);
\draw [fill=black] (5,0) circle (1.5pt);

\draw [fill=black] (6,0) circle (1.5pt);
\draw [fill=black] (7,0) circle (1.5pt);
 \draw [fill=black] (5,0.8) circle (1.5pt);
\draw [fill=black] (6,-0.8) circle (1.5pt);
\draw [fill=black] (6,0.8) circle (1.5pt);

\draw [fill=black] (9,0) circle (1.5pt);
\draw [fill=black] (10,0) circle (1.5pt);
\draw [fill=black] (11,0) circle (1.5pt);
\draw [fill=black] (12,0) circle (1.5pt);
\draw [fill=black] (13,0) circle (1.5pt);

\draw [fill=black] (12,0.8) circle (1.5pt);
\draw [fill=black] (10,0.8) circle (1.5pt);
\draw [fill=black] (11,0.8) circle (1.5pt);
\draw [fill=black] (10,-0.8) circle (1.5pt);
\draw [fill=black] (12,-0.8) circle (1.5pt);

\end{scriptsize}
\end{tikzpicture}
\caption{Chemical trees $H_1$, $H_2$ and $H_3$.}\label{6}
\end{center}
\end{figure}
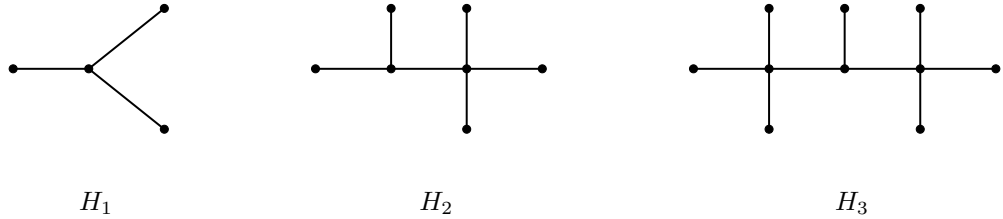

Let $\Gamma'_3$ denote the subclass of chemical trees in $\Gamma_3$ that contains exactly one vertex of degree three, which is adjacent only to vertices of degree four, that is, $x_{14}=\frac{2\,n+1}{3}$, $x_{34}=3$, $x_{44}=\frac{n-13}{3}$, and $x_{11}=x_{12}=x_{13}=x_{22}=x_{23}=x_{24}=x_{33}=0$. For $T\in \Gamma'_3$, we have
\begin{align}
ASO(T)=\frac{\sqrt{51}}{9}\,(2\,n+1)+\frac{4\,\sqrt{3}}{9}\,(n-13)+3\,\sqrt{5}.\nonumber
\end{align}

\begin{theorem}\label{theorem2.7}
Let $n$ be an integer with $n\geq 13$ and $n\equiv 1\pmod 3$. Then
\begin{align}
ASO(T)\leq \frac{\sqrt{51}}{9}(2\,n+1)+\frac{4\,\sqrt{3}}{9}(n-13)+3\,\sqrt{5}\label{e1}
\end{align}
with equality if and only if $T\in \Gamma'_3$.
\end{theorem}

\begin{proof} Let $T\in \Gamma_1$. Then $n_2=n_3=0$. Since $n\equiv 1\pmod 3$ and $n\geq 13$, we can assume that $n=3\,k+1$, where $k$ is an integer with $k\geq 4$. By the well-known Handshaking Lemma, the sum of the degrees of all vertices in the tree $T$ is given by
$2\,(n-1)=n_1+4\,n_4$ as $n_2=n_3=0$. On the other hand, the order of $T$ satisfies $n=n_1+n_4$ as $n_2=n_3=0$. Solving this system of equations yields $$n_1=\frac{2\,(n+1)}{3}=2\,k+\frac{4}{3}~~\mbox{ and }~~n_4=\frac{n-2}{3}=k-\frac{1}{3}.$$
This is impossible, since both $n_1$ and $n_4$ must be integers. Hence, we arrive at a contradiction.

\vspace*{3mm}

Let $T\in \Gamma_2$. Then $n_2=1$ and $n_3=0$. By the well-known Handshaking Lemma, we have $2\,(n-1)=n_1+2+4\,n_4$. Moreover, $n=n_1+1+n_4$. Solving these two equations, we obtain
$$n_1=\frac{2\,n}{3}=2\,k+\frac{2}{3}~~\mbox{ and }~~n_4=\frac{n-3}{3}=k-\frac{2}{3}.$$
This is impossible, since both $n_1$ and $n_4$ must be integers. Hence, we arrive at a contradiction.

\vspace*{3mm}

Therefore, we conclude that $T\notin\Gamma_1\cup\Gamma_2$. Now, consider a function
\begin{align}
    f(i,j)=\frac{4}{3}\left[\frac{4\,\sqrt{3}}{3}\left(\frac{1}{i}+\frac{1}{j}-\frac{5}{4}\right)-\frac{\sqrt{51}}{3}\left(\frac{1}{i}+\frac{1}{j}-\frac{1}{2}\right)+\frac{3}{4}\sqrt{\frac{i^2+j^2}{i+j-2}}\right]=\frac{4}{3}\,\phi_1(i,j),\mbox{~~}(i,j)\in A_1.\nonumber
\end{align}
 By directly calculating, we find that $f(1,2)<0$, $f(1,3)<0$, $f(2,2)<0$, $f(2,3)<0$, $f(2,4)<0$, $f(3,3)<0$, $f(3,4)<0$. Furthermore, the maximum value of $f(i,j)$ over all $(i,j)\in A_1$ occurs at 
 $f(3,4)$.
From (\ref{1kinkar1}), we have
\begin{align}
ASO(T)
&=\frac{\sqrt{51}}{3}\cdot\frac{2\,n+2}{3}+\frac{4\,\sqrt{3}}{3}\cdot\frac{n-5}{3}+\sum_{(i,j)\in A_1}\,f(i,j)\,x_{ij}\nonumber\\[2mm]
&=\frac{\sqrt{51}}{3}\cdot\frac{2\,n+2}{3}+\frac{4\,\sqrt{3}}{3}\cdot\frac{n-5}{3}+f(1,2)x_{12}+f(1,3)x_{13}+f(2,2)x_{22}+f(2,3)x_{23}\nonumber\\[2mm]
&~~~~~~~~~~~~~~~~~~~~~~~~~~~+f(2,4)x_{24}+f(3,3)x_{33}+f(3,4)x_{34}. \label{inequality2411}
\end{align}

\vspace*{3mm}

\noindent
${\bf Case\,1.}$ $n_3=0$. By Hand-shaking lemma and the construction of chemical tree, we have
\begin{align}    n_1+n_2+n_4=n~~~\textup{and}~~~n_1+2\,n_2+4\,n_4=2\,(n-1).\nonumber
\end{align}
Then $n_2+3\,n_4=n-2=(3\,k+1)-2=3\,k-1$ for some integer $k\geq 4$, that is, $n_2=3\,(k-n_4)-1$. If $k\leq n_4$, then $n_2\leq -1$, a contradiction. Otherwise, $k\geq n_4+1$. Thus we have $n_2\geq 2$.

\vspace*{3mm}

Since $n_3=0$, it follows $x_{13}=x_{23}=x_{33}=x_{34}=0$. Using this, from (\ref{inequality2411}), we obtain
\begin{align}
  ASO(T)&=\frac{\sqrt{51}}{3}\cdot\frac{2\,n+2}{3}+\frac{4\,\sqrt{3}}{3}\cdot\frac{n-5}{3}+f(1,2)x_{12}+f(2,2)x_{22}+f(2,4)x_{24}\nonumber\\[2mm]
  &\leq \frac{\sqrt{51}}{3}\cdot\frac{2\,n+2}{3}+\frac{4\,\sqrt{3}}{3}\cdot\frac{n-5}{3}+f(2,4)\left(x_{12}+2\,x_{22}+x_{24}\right)\nonumber
  \end{align}
  \begin{align}
  &=\frac{\sqrt{51}}{3}\cdot\frac{2\,n+2}{3}+\frac{4\,\sqrt{3}}{3}\cdot\frac{n-5}{3}+2\,n_2f(2,4)\nonumber\\[2mm]
  &\leq \frac{\sqrt{51}}{3}\cdot\frac{2\,n+2}{3}+\frac{4\,\sqrt{3}}{3}\cdot\frac{n-5}{3}+4\,f(2,4)\nonumber\\[2mm]
  &=\frac{\sqrt{51}}{3}\cdot\frac{2\,n+2}{3}+\frac{4\,\sqrt{3}}{3}\cdot\frac{n-5}{3}+\frac{16}{3}\left(-\frac{2\,\sqrt{3}}{3}-\frac{\sqrt{51}}{12}+\frac{3\,\sqrt{5}}{4}\right)\nonumber\\[2mm]
  &=\frac{\sqrt{51}}{9}\,(2\,n-2)+\frac{4\sqrt{3}}{9}\,(n-13)+4\,\sqrt{5}\nonumber\\[2mm]
  &<\frac{\sqrt{51}}{9}\,(2\,n+1)+\frac{4\,\sqrt{3}}{9}(n-13)+3\,\sqrt{5}\nonumber
\end{align}
as $f(1,2)<f(2,4)$, $f(2,2)<2\,f(2,4)$.
The upper bound strictly holds.

\vspace*{3mm}

\noindent
${\bf Case\,2.}$ $n_3\geq 1$. Since $f(2,2)<f(1,2)<f(2,4)<0$, $f(3,4)>f(1,3)>f(2,3)$ and $f(3,3)<2\,f(3,4)$, from (\ref{inequality2411}), we obtain
\begin{align}
    ASO(T)&\leq \frac{\sqrt{51}}{3}\cdot\frac{2\,n+2}{3}+\frac{4\,\sqrt{3}}{3}\cdot\frac{n-5}{3}+f(3,4)(x_{13}+x_{23}+x_{34})+f(3,3)x_{33}\label{e2}\\[2mm]
&\leq \frac{2\,\sqrt{51}}{9}\,(n+1)+\frac{4\,\sqrt{3}}{9}\,(n-5)+f(3,4)(x_{13}+x_{23}+2\,x_{33}+x_{34})\label{e3}\\[2mm]
&=\frac{\sqrt{51}}{3}\cdot\frac{2n+2}{3}+\frac{4\sqrt{3}}{3}\cdot\frac{n-5}{3}+3\,n_3\,f(3,4)\nonumber\\[2mm]
&\leq \frac{\sqrt{51}}{3}\cdot\frac{2n+2}{3}+\frac{4\sqrt{3}}{3}\cdot\frac{n-5}{3}+3\,f(3,4)\label{e4}\\[2mm]
&=\frac{\sqrt{51}}{3}\left(\frac{2n+2}{3}-\frac{1}{3}\right)+\frac{4\sqrt{3}}{3}\left(\frac{n-5}{3}-\frac{8}{3}\right)+3\,\sqrt{5}\nonumber\\[2mm]
&=\frac{\sqrt{51}}{9}\,(2\,n+1)+\frac{4\,\sqrt{3}}{9}\,(n-13)+3\,\sqrt{5}\nonumber
\end{align}
as $n_3\geq 1$ and
$$3\,f(3,4)=-\frac{32\,\sqrt{3}}{9}-\frac{\sqrt{51}}{9}+3\,\sqrt{5}<0.$$
The first part of the proof is done.

\vspace*{3mm}

Suppose that equality holds in (\ref{e1}). Then all inequalities in ${\bf Case\,2}$ must be equalities. From the equality in (\ref{e2}), we obtain $x_{12}=x_{13}=x_{22}=x_{23}=x_{24}=0$. From the equality in (\ref{e3}), we obtain $x_{33}=0$. Again from the equality in (\ref{e4}), we obtain $n_3=1$. Thus we have $x_{34}=x_{13}+x_{23}+2\,x_{33}+x_{34}=3\,n_3=3$ and $n_2=\frac{1}{2}\,(x_{12}+2\,x_{22}+x_{23}+x_{24})=0$.
Thus we have
\begin{align}    n_1+n_2+n_3+n_4=n~~~\textup{and}~~~n_1+2\,n_2+3\,n_3+4\,n_4=2\,(n-1),\nonumber
\end{align}
that is,
$$n_1+n_4=n-1~~~\textup{and}~~~n_1+4\,n_4=2\,n-5,$$
that is,
  $$n_1=\frac{2\,n+1}{3}~~\mbox{ and }~~n_4=\frac{n-4}{3}.$$
Thus we have
$$x_{14}=x_{12}+x_{13}+x_{14}=n_1=\frac{2\,n+1}{3}$$
Since $4\,n_4=x_{14}+x_{24}+x_{34}+2\,x_{44}=x_{14}+3+2\,x_{44}$, we obtain
$$x_{44}=\frac{1}{2}\,(4\,n_4-x_{14}-3)=\frac{n-13}{3}.$$
Hence $x_{14}=\frac{2\,n+1}{3}$, $x_{34}=3$, $x_{44}=\frac{n-13}{3}$, and $x_{11}=x_{12}=x_{13}=x_{22}=x_{23}=x_{24}=x_{33}=0$, that is, $T\in\Gamma_3'$.

\vspace*{3mm}

Conversely, let $T\in\Gamma_3'$. Then we have $x_{14}=\frac{2\,n+1}{3}$, $x_{34}=3$, $x_{44}=\frac{n-13}{3}$, and $x_{11}=x_{12}=x_{13}=x_{22}=x_{23}=x_{24}=x_{33}=0$. Hence
$$ASO(T)=\frac{\sqrt{51}}{9}\,(2\,n+1)+\frac{4\,\sqrt{3}}{9}\,(n-13)+3\,\sqrt{5}.$$
\end{proof}
The maximal chemical trees with respect to $ASO$ for $n=6$ are easily obtained by using Sage~\cite{ST}, as shown in Figure \ref{7}.
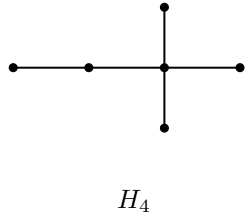
\begin{figure}[h]
\begin{center}
\begin{tikzpicture}
\clip(-2.25,-2.5) rectangle (16,1.8);

\draw [line width=0.75pt] (4,0)--(7,0);
  \draw [line width=0.75pt] (6,-0.8)--(6,0.8);



\draw (5.25,-1.5) node [anchor=north west]{$H_4$};

\begin{scriptsize}

\draw [fill=black] (4,0) circle (1.5pt);
\draw [fill=black] (5,0) circle (1.5pt);

\draw [fill=black] (6,0) circle (1.5pt);
\draw [fill=black] (7,0) circle (1.5pt);
\draw [fill=black] (6,-0.8) circle (1.5pt);
\draw [fill=black] (6,0.8) circle (1.5pt);

\end{scriptsize}
\end{tikzpicture}
\caption{Chemical tree $H_4$.}\label{7}
\end{center}
\end{figure}

Let $\Gamma'_2$ denote the subclass of chemical trees in $\Gamma_2$ that contains exactly one vertex of degree two, which is adjacent only to vertices of degree four, that is, $x_{14}=\frac{2\,n}{3}$, $x_{24}=2$, $x_{44}=\frac{n-9}{3}$, and $x_{11}=x_{12}=x_{13}=x_{22}=x_{23}=x_{34}=x_{33}=0$. For $T\in \Gamma'_2$, we have
\begin{align}
ASO(T)=\frac{2\,\sqrt{51}}{9}n+\frac{4\,\sqrt{3}}{9}(n-9)+2\,\sqrt{5}.\nonumber
\end{align}

\begin{theorem}\label{theorem3.4}
    Let $n$ be an integer with $n\geq 9$ and $n\equiv 0\pmod{3}$. Then,
    \begin{align}
        ASO(T)\leq \frac{2\,\sqrt{51}}{9}n+\frac{4\,\sqrt{3}}{9}(n-9)+2\,\sqrt{5}\label{e22}
    \end{align}
with equality if and only if $T\in\Gamma_2^\prime$.

\end{theorem}
\begin{proof}
Let $T\in \Gamma_1$. Since $n\equiv 0\pmod 3$ and $n\geq 9$, we can assume that $n=3\,k$, where $k$ is an integer with $k\geq 3$. By the well-known Handshaking lemma, the sum of the degrees of all vertices in the tree $T$ is given by $2\,(n-1)=n_1+4\,n_4$ as $T\in \Gamma_1$. On the other hand, the order of $T$ satisfies $n=n_1+n_4$. Solving this system of equations yields
\begin{align}
    n_1=2\,k+\frac{2}{3}~~\textup{and}~~n_4=k-\frac{2}{3}.
\end{align}
This is impossible since both $n_1$ and $n_4$ must be integers. Hence, we arrive at a contradiction.

\vspace*{3mm}

Let $T\in\Gamma_3$. By the well-known Handshaking Lemma, the sum of the degrees of all vertices in the tree $T$ is given by $2\,(n-1)=n_1+3+4\,n_4$. Moreover, $n=n_1+1+n_4$. Solving these two equations, we obtain
\begin{align}
n_1=2\,k+\frac{1}{3}~\textup{and}~~n_4=k-\frac{4}{3}. \nonumber
\end{align}
This is impossible as both $n_1$ and $n_4$ must be integers. Hence, we arrive at a contradiction.

\vspace*{3mm}

Therefore, we conclude $T\notin\Gamma_1\cup\Gamma_3$. Again, consider the function,
\begin{align}
    f(i,j)=\frac{4}{3}\left[\frac{4\sqrt{3}}{3}\left(\frac{1}{i}+\frac{1}{j}-\frac{5}{4}\right)-\frac{\sqrt{51}}{3}\left(\frac{1}{i}+\frac{1}{j}-\frac{1}{2}\right)+\frac{3}{4}\sqrt{\frac{i^2+j^2}{i+j-2}}\right]=\frac{4}{3}\,\phi_1(i,j),~~(i,j)\in A_1.\nonumber
\end{align}
 By direct calculation, we have $f(1,2)<0$, $f(1,3)<0$, $f(2,2)<0$, $f(2,3)<0$, $f(2,4)<0$, $f(3,3)<0$, $f(3,4)<0$ and $\{f(i,j):(i,j)\in A_1\setminus (3,4)\}_{\textup{max}}=f(2,4)$. From (\ref{1kinkar1}), we have
 \begin{align}
    & ASO(T)=\frac{\sqrt{51}}{3}\cdot\frac{2n+2}{3}+\frac{4\sqrt{3}}{3}\cdot\frac{n-5}{3}+f(1,2)x_{12}+f(1,3)x_{13}+f(2,2)x_{22}+f(2,3)x_{23}+f(2,4)x_{24}\nonumber\\[2mm]
     &~~~~~~~~~~~~~~~~~~~~~~~~~~~+f(3,3)x_{33}+f(3,4)x_{34}.\label{e23}
 \end{align}

\vspace*{3mm}

\noindent
${\bf Case\,1.}$ $n_2=0$. By Hand-shaking Lemma and the construction of chemical tree, we have
\begin{align}
    n_1+n_3+n_4=n~~~\textup{and}~~~n_1+3\,n_3+4\,n_4=2\,(n-1).\nonumber
\end{align}
Then $n_3=\frac{n-3\,n_4-2}{2}=\frac{3\,(k-n_4)-2}{2}\in\mathbb{N}$ for some integer $k\geq 3$. It follows that $k-n_4\geq 2$. Hence, $n_3\geq 2$.

\vspace*{3mm}

Since $n_2=0$, we have $x_{12}=x_{22}=x_{23}=x_{24}=0$. Using this, from (\ref{e23}), we obtain
\begin{align}
  ASO(T)&=\frac{\sqrt{51}}{3}\cdot\frac{2\,n+2}{3}+\frac{4\,\sqrt{3}}{3}\cdot\frac{n-5}{3}+f(1,3)x_{13}+f(3,3)x_{33}+f(3,4)x_{34}\nonumber\\[2mm]
  &\leq \frac{\sqrt{51}}{3}\cdot\frac{2\,n+2}{3}+\frac{4\,\sqrt{3}}{3}\cdot\frac{n-5}{3}+f(3,4)\left(x_{13}+2\,x_{33}+x_{34}\right)\nonumber\\[2mm]
  &=\frac{\sqrt{51}}{3}\cdot\frac{2\,n+2}{3}+\frac{4\,\sqrt{3}}{3}\cdot\frac{n-5}{3}+3\,n_3f(3,4)\nonumber\\[2mm]
  &\leq \frac{\sqrt{51}}{3}\cdot\frac{2\,n+2}{3}+\frac{4\sqrt{3}}{3}\cdot\frac{n-5}{3}+6\,f(3,4)\nonumber\\
  &=\frac{2\,\sqrt{51}}{9}\,(n+1)+\frac{4\,\sqrt{3}}{9}\,(n-5)+8\left(-\frac{8\,\sqrt{3}}{9}-\frac{\sqrt{51}}{36}+\frac{3\,\sqrt{5}}{4}\right)\nonumber
  \end{align}
  \begin{align}
    &=\frac{2\,\sqrt{51}}{9}\,n+\frac{4\,\sqrt{3}}{9}\,(n-21)+6\,\sqrt{5}\nonumber\\[2mm]
    &<\frac{2\,\sqrt{51}}{9}\,n+\frac{4\,\sqrt{3}}{9}\,(n-9)+2\,\sqrt{5}\nonumber
\end{align}
as $f(1,3)<f(3,4)$, $f(3,3)<2\,f(3,4)$.
The upper bound strictly holds.

\vspace*{3mm}

\noindent
${\bf Case\,2.}$ $n_2\geq 1$. Since $f(3,3)<f(1,3)<f(3,4)<0$, $f(2,4)>f(1,2)>f(2,3)$ and $f(2,2)<2\,f(2,4)$, from (\ref{e23}), we obtain
\begin{align}
    ASO(T)&\leq \frac{\sqrt{51}}{3}\cdot\frac{2\,n+2}{3}+\frac{4\,\sqrt{3}}{3}\cdot\frac{n-5}{3}+f(2,4)(x_{12}+x_{23}+x_{24})+f(2,2)x_{22}\label{e24}\\[2mm]
     &\leq \frac{2\,\sqrt{51}}{9}(n+1)+\frac{4\,\sqrt{3}}{9}(n-5)+f(2,4)(x_{12}+2\,x_{22}+x_{23}+x_{24})\label{e25}\\[2mm]
&=\frac{\sqrt{51}}{3}\cdot\frac{2\,n+2}{3}+\frac{4\,\sqrt{3}}{3}\cdot\frac{n-5}{3}+2\,n_2\,f(2,4)\nonumber\\[2mm]
&\leq \frac{\sqrt{51}}{3}\cdot\frac{2\,n+2}{3}+\frac{4\,\sqrt{3}}{3}\cdot\frac{n-5}{3}+2\,f(2,4)\label{e26}
\end{align}
\begin{align}
&=\frac{\sqrt{51}}{3}\left(\frac{2\,n+2}{3}-\frac{2}{3}\right)+\frac{4\,\sqrt{3}}{3}\left(\frac{n-5}{3}-\frac{4}{3}\right)+2\sqrt{5}\nonumber\\[2mm]
    &=\frac{2\,\sqrt{51}}{9}\,n+\frac{4\,\sqrt{3}}{9}\,(n-9)+2\,\sqrt{5}\nonumber
\end{align}
as $n_2\geq 1$ and
$$2\,f(2,4)=-\frac{16\,\sqrt{3}}{9}-\frac{2\,\sqrt{51}}{9}+2\,\sqrt{5}<0.$$
The first part of the proof is done.

\vspace*{3mm}

Suppose that equality holds in (\ref{e22}). Then all inequalities in ${\bf Case\,2}$ must be equalities. From the equality in (\ref{e24}), we obtain $x_{12}=x_{13}=x_{23}=x_{33}=x_{34}=0$. From the equality in (\ref{e25}), we obtain $x_{22}=0$. Again from the equality in (\ref{e26}), we obtain $n_2=1$. Thus we have $x_{24}=x_{12}+2\,x_{22}+x_{23}+x_{24}=2\,n_2=2$ and $n_3=\frac{1}{3}\,(x_{13}+x_{23}+2\,x_{33}+x_{34})=0$.
From $n_1+n_2+n_3+n_4=n$ and $n_1+2\,n_2+3\,n_3+4\,n_4=2\,(n-1)$, we have
$n_1+n_4=n-1$ and $n_1+4\,n_4=2\,n-4$, that is,
$n_1=\frac{2\,n}{3}$ and $n_4=\frac{n-3}{3}$.
Thus, we have
$$x_{14}=x_{12}+x_{13}+x_{14}=n_1=\frac{2\,n}{3},~\mbox{ and }~ x_{44}=\frac{1}{2}\,(4n_4-x_{14}-x_{24}-x_{34})=\frac{n-9}{3}.
$$
Hence $x_{14}=\frac{2\,n}{3}$, $x_{24}=2$, $x_{44}=\frac{n-9}{3}$, and $x_{11}=x_{12}=x_{13}=x_{22}=x_{23}=x_{33}=x_{34}=0$, that is, $T\in\Gamma_2'$.

\vspace*{3mm}

Conversely, let $T\in\Gamma_2'$. Then we have $x_{14}=\frac{2\,n}{3}$, $x_{24}=2$, $x_{44}=\frac{n-9}{3}$, and $x_{11}=x_{12}=x_{13}=x_{22}=x_{23}=x_{33}=x_{34}=0$. Hence
$$ASO(T)=\frac{2\,\sqrt{51}}{9}\,n+\frac{4\,\sqrt{3}}{9}\,(n-9)+2\sqrt{5}.$$
\end{proof}

\begin{table}[ht]
\begin{center}
\begin{tabular}{|c|c|c|c|c|c|c|c|c|c|}
\hline
$(d_i,d_j)$ & $(1,2)$ & (1,3) & (1,4) & $(2,2)$ & $(2,3)$ & $(2,4)$ & $(3,3)$ & $(3,4)$ & $(4,4)$ \\
\hline
$h(d_i,d_j)$ & $\sqrt{5}$ & $\sqrt{5}$ & $\sqrt{\frac{17}{3}}$ & $2$ & $\sqrt{\frac{13}{3}}$ & $\sqrt{5}$ & $\frac{3}{\sqrt{2}}$ & $\sqrt{5}$ & $\frac{4}{\sqrt{3}}$ \\
\hline
\end{tabular}
\caption{Values of the function $h(d_i,d_j)$ for different edges $v_i v_j \in E(\Omega)$.
}\label{table2}
\end{center}
\end{table}


\begin{proposition}
    For any chemical graph $\Omega\in \mathcal{CG}_n$,
    \begin{align}
       2\,(n-3)+2\,\sqrt{5} \leq ASO(\Omega)\leq \frac{8\,n\sqrt{3}}{3},\label{e44}
    \end{align}
   herein, equality is attained on the left side if and only if $\Omega\cong P_n$, while equality is attained on the right side if and only if $\Omega$ is a $4$-regular graph. 
\end{proposition}
\begin{proof}
\noindent\textbf{Lower Bound:}
Since every chemical graph is a connected graph, by Theorem~\ref{e444} we have
$$
ASO(\Omega) \ge 2\,\sqrt{5} + 2\,(n-3).
$$
Moreover, equality holds if and only if $\Omega \cong P_n$.

\vspace*{3mm}

\noindent\textbf{Upper Bound:} Let $p\,(\geq 0)$ be the number of pendant vertices in $\Omega$. Since $\Omega$ is a chemical graph, we have $2\,m = \sum_{i=1}^{n} d_i \le p+4\,(n-p)$, that is,
$m \le 2\,n-\frac{3\,p}{2}$. Moreover, the equality holds if and only if all the non-pendant vertices are of degree $4$. From Table \ref{table2}, one can easily see that $h(d_i,d_j)\leq \sqrt{\frac{17}{3}}$ for any edge $v_iv_j\in E(\Omega)$ with equality if and only if $(d_i,d_j)=(4,1)$, and $h(d_i,d_j)\leq \frac{4}{\sqrt{3}}$ for all non-pendant edges $v_iv_j\in E(\Omega)$ with equality if and only if $(d_i,d_j)=(4,4)$. Thus we obtain
\begin{align*}
ASO(\Omega)
= \sum_{v_iv_j\in E(\Omega)}\,\sqrt{\frac{d_i^2+d_j^2}{d_i+d_j-2}}&=\sum_{v_iv_j\in E(\Omega)}\,h(d_i,d_j)\\[3mm]
&=\sum_{v_iv_j\in E(\Omega),\atop d_i\geq d_j=1}\,h(d_i,d_j)+\sum_{v_iv_j\in E(\Omega),\atop d_i\geq d_j\geq 2}\,h(d_i,d_j)\\[3mm]
&\le \sqrt{\frac{17}{3}}\,p+\frac{4}{\sqrt{3}}\,(m-p)\\[3mm]
&\leq\frac{8\,n}{\sqrt{3}}+\Big(\sqrt{\frac{17}{3}}-\frac{10}{\sqrt{3}}\Big)\,p
\end{align*}
as $m \le 2\,n-\frac{3\,p}{2}$. If $p\geq 1$, then from the above, we obtain
$$ASO(\Omega)\leq \frac{8\,n}{\sqrt{3}}+\Big(\sqrt{\frac{17}{3}}-\frac{10}{\sqrt{3}}\Big)<\frac{8\,n}{\sqrt{3}}.$$
The upper bound strictly holds. Otherwise, $p=0$.
Thus we have
$$ASO(\Omega)\leq \frac{8\,n}{\sqrt{3}}.$$
Moreover, the equality holds if and only if all the non-pendant vertices are of degree $4$, that is, if and only if $\Omega$ is isomorphic to a $4$-regular graph (as $p=0$). This completes the proof.
\end{proof}

\section{Chemical Application of the Augmented Sombor Index}\label{section4}
Topological indices are key tools in mathematical chemistry; yet, many are mostly theoretical, without direct chemical interpretation. It is therefore essential to choose a descriptor that provides substantial structural information. Consequently, it needs careful examination, as an effective index should correlate well with the experimental physicochemical properties of molecules. In \cite{DGA}, the researchers revealed that the $ASO$ index can predict various physicochemical properties of octane isomers. However, these molecules consist solely of C–C single bonds, and their structural simplicity limits the assessment of the broader applicability of the index. To address this limitation, we consider molecules with cyclic structures and double bonds, specifically a set of benzenoid hydrocarbons (BHCs) and several anti-cancer drugs in this study.

To begin with, we analyze the relationship of the index with two physicochemical properties of BHCs: the boiling point ($BP$) and the $\pi$-electron energy ($E_{\pi}$). The experimental data were taken from \cite{mon23,lucic09,NJC1,dasalharbi25}. Figure \ref{BH} displays the set of BHCs examined in this work. The relationship between the index and each property is examined through the following linear regression model.
\begin{equation}\label{EQN1}
\text{Property} = \text{Slope} \ (\pm 2 \times \text{error}) \text{Descriptor} + \text{Intercept} \ (\pm 2 \times \text{error}).
\end{equation}

In addition to the correlation coefficient $R$, the models are evaluated using the $F$-test ($F$), root mean square error ($RMSE$), and significance F ($SF$), to provide a more accurate assessment of model fit.
For $BP$ and $E_\pi$, equation \eqref{EQN1}
provides the following models.
\begin{eqnarray}\label{EQN10}
&BP = 9.4444(\pm 0.3786)(ASO)+21.7895(\pm 18.9467),\\
\nonumber
&R^{2}=0.9928,~~~RMSE=8.2126,~~~F=2488.5585
,~~~SF=9.45 \times 10^{-21}.
\end{eqnarray}
\begin{eqnarray}\label{EQN11}
&E_\pi = 0.5326(\pm 0.0135)(ASO)+1.9443
(\pm 0.6741),\\
\nonumber
&R^{2}=0.9971
,~~~RMSE=0.2922,~~~F=6253.0485
,~~~SF=2.46 \times 10^{-24}.
\end{eqnarray}
\vspace*{1mm}

The models~\eqref{EQN10} and \eqref{EQN11} show that the $ASO$ index  accounts for 99\% of the variation in both $BP$ and $E_{\pi}$. These relationships are illustrated in Figure~\ref{dhvap S.jpg}. In addition, the combination of low $RMSE$, very high $F$ statistics and very small significance levels ($SF < 0.05$) confirms the conclusion that the $ASO$ index provides accurate and reliable predictions for $BP$ and $E_{\pi}$.

\vspace*{3.5mm}
\noindent
We next study molecular graphs that contain cyclic substructures, as such structures are often found in drug molecules. The compounds analyzed are Carmustine, Deguelin, Convolutamide A, Tambjamine K, Aminopterin, Melatonin, Convolutamine F, Perfragilin A, Caulibugulone E, Minocycline,
Convolutamydine A, Aspidostomide E, Amathaspiramide E, and Pterocellin B. For these compounds, we examine two properties: molar refraction ($MR$) and boiling point ($BP$). The experimental data were taken from \cite{dasalharbi25,Heliyon2020}.

\begin{figure}[H]

\centering


\caption{Graph-based representations of BHCs.}

\label{BH}

\end{figure}

\vspace*{5mm}



\begin{figure}[H]
\includegraphics[width=15cm]{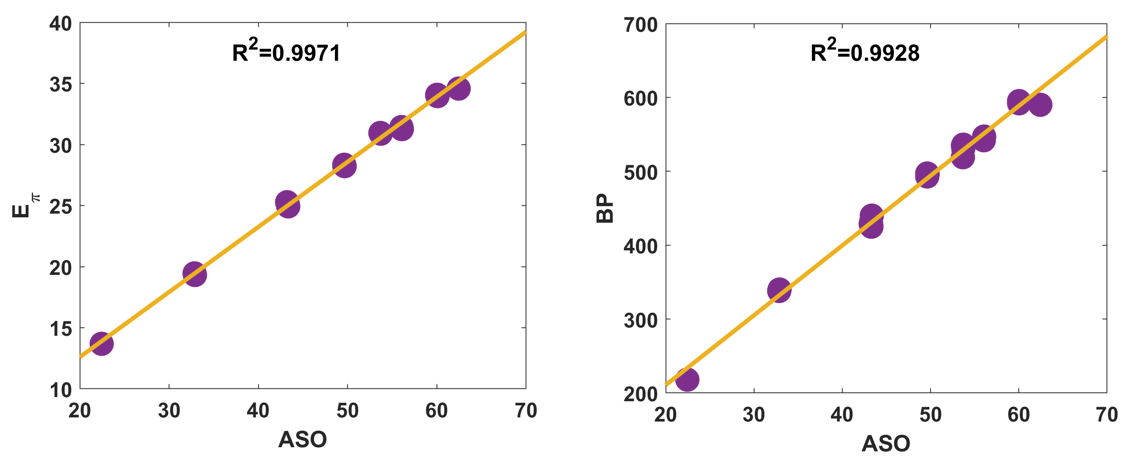}
\centering
\caption{Linear regression plots of $ASO$ with $BP$, $E_{\pi}$ for BHCs.}
\label{dhvap S.jpg}
\end{figure}
\vspace*{3mm}
\begin{figure}[H]
\includegraphics[width=15cm]{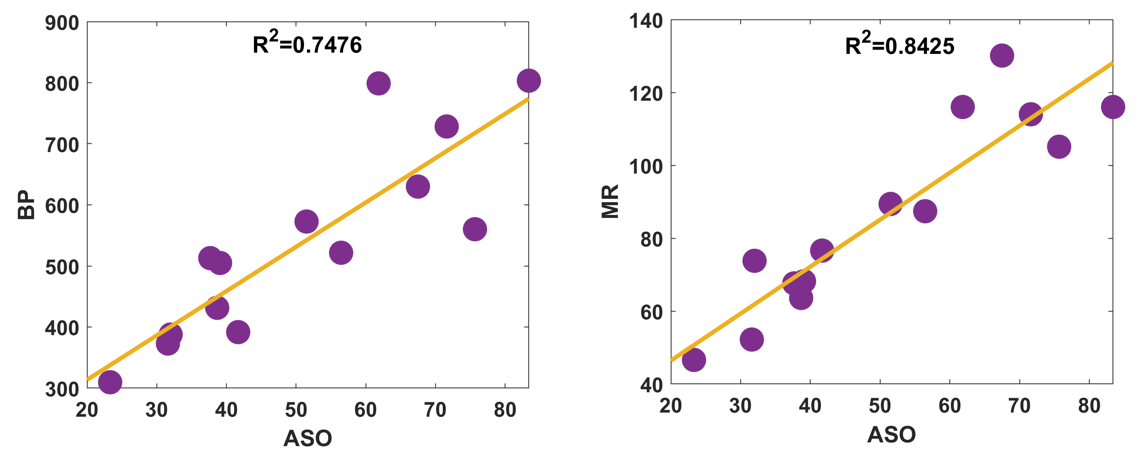}
\centering \caption{Linear regression plots of $ASO$ with $MR$, $BP$ for
drug-related compounds.} \label{MR and E}
\end{figure}

\noindent
\vspace*{0.3mm}

 For $MR$ and $BP$, equation \eqref{EQN1} provides
the following models.
\begin{eqnarray}\label{EQN12}
&BP = 7.2566(\pm 2.4344)(ASO)+168.5229
(\pm 131.3719),\\
\nonumber
&R^{2}=0.7476
,~~~RMSE=76.1290,~~~F=35.5432
,~~~SF=6.59 \times 10^{-5},
\end{eqnarray}
\begin{eqnarray}\label{EQN13}
&MR = 1.2886(\pm 0.3217)(ASO)+20.6520
(\pm 17.3624),\\\nonumber
&R^{2}=0.8425
,~~~RMSE=10.0614,~~~F=64.1693
,~~~SF=3.71 \times 10^{-6}.
\end{eqnarray}

The models~\eqref{EQN12} and \eqref{EQN13} indicate that the $ASO$ index accounts for 75\% of the variation in $BP$ and 84\% of the variation in $MR$. These relationships are illustrated in Figure~\ref{MR and E}. Other evaluated parameters also support the conclusion that the $ASO$ index effectively predicts both $BP$ and $MR$.  A closer comparison between $BP$ and $MR$ suggests that the index predicts $MR$ more accurately. In Figure~\ref{MR and E}, the pink circles for $MR$ lie closer to the orange regression line than those for $BP$, and the associated $RMSE$ for $MR$ is lower while the $F$ value is higher. Therefore, the $ASO$ index demonstrates better predictive performance for $MR$ than for $BP$.

\vspace{2mm}

\subsection*{Comparative Analysis}

This subsection compares the $ASO$ index with several of its variants, including $SO$ (Sombor index), $ESO$ (elliptic Sombor index), and $EU$ (Euler Sombor index). For this comparison, the indices are correlated with selected physicochemical properties of BHCs and molecular graphs that are relevant to drug design.

The absolute correlation coefficients of $ASO$, $SO$, $ESO$, and $EU$ with the physicochemical properties $BP$ and $E_\pi$ for a dataset of BHCs are displayed in Figures~\ref{fig:DHVAP33.png} and \ref{fig:HVAP33.png}. Similarly, the absolute correlation coefficients between these indices and the physicochemical properties $BP$ and $MR$ for a dataset of molecular graphs associated with drug design are shown in Figures~\ref{fig:DHVAP44.png} and \ref{fig:HVAP44.png}.

For the set of BHCs, considering both $BP$ and $E_\pi$, it is observed from Figures~\ref{fig:DHVAP33.png} and \ref{fig:HVAP33.png} that the absolute correlation coefficients of the studied topological indices follow the order
\[
ASO > SO > EU > ESO.
\]
In contrast, for the selected molecular graphs related to drug design, considering both $BP$ and $MR$, it is observed from Figures~\ref{fig:DHVAP44.png} and \ref{fig:HVAP44.png} that the corresponding order of correlation strength is
\[
ASO> EU > SO> ESO.
\]

These observations indicate that the $ASO$ index exhibits the strongest correlation, suggesting that it may more effectively reflect the structural characteristics associated with the physicochemical properties of both BHCs and drug-relevant molecular graphs.

\vspace*{0.9mm}

\begin{figure}[H]
\includegraphics[width=8cm]{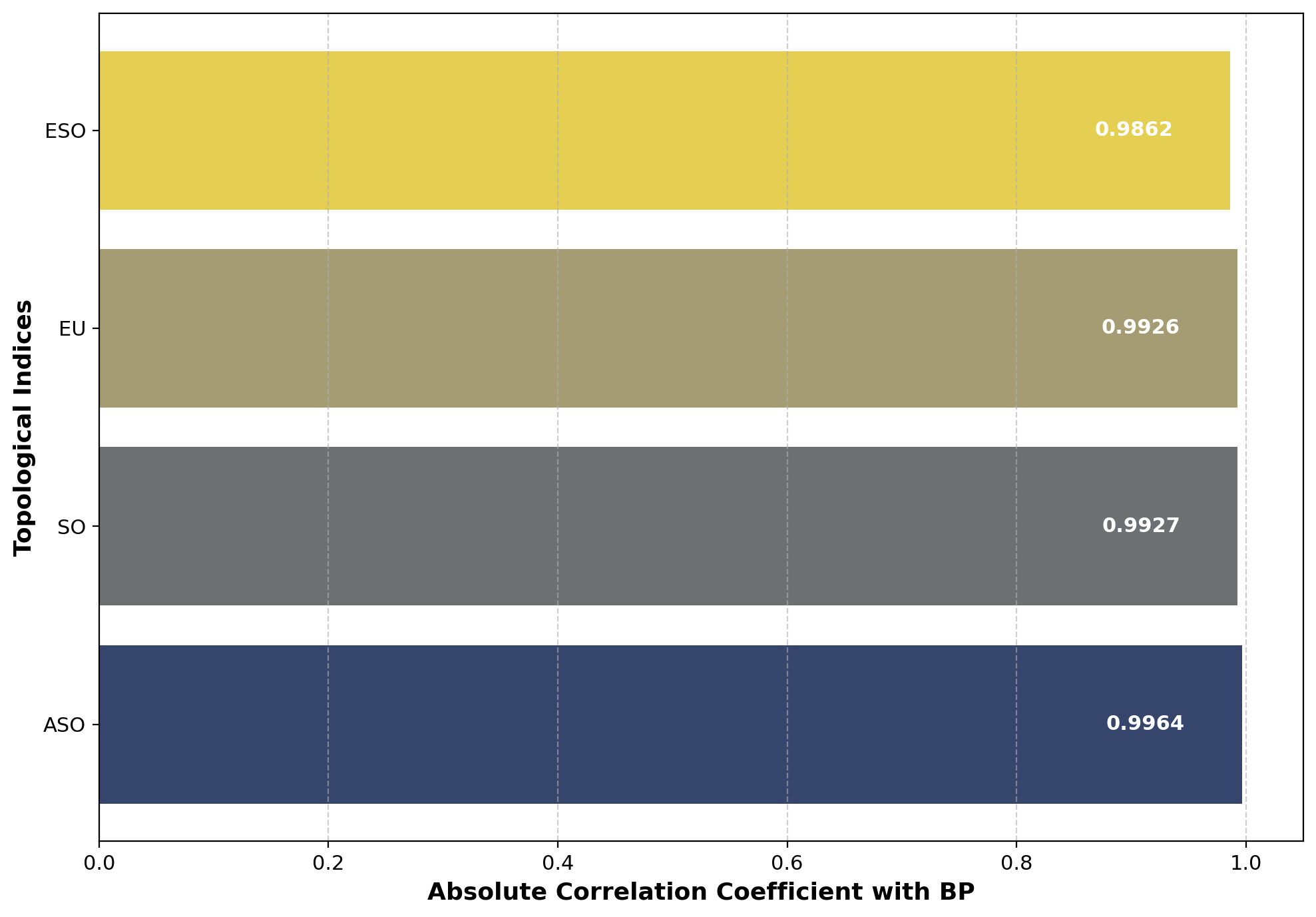}
\centering \caption{Absolute correlation plot: $BP$ versus selected topological indices.} \label{fig:DHVAP33.png}
\end{figure}
\vspace*{2mm}
\begin{figure}[H]
\includegraphics[width=8cm]{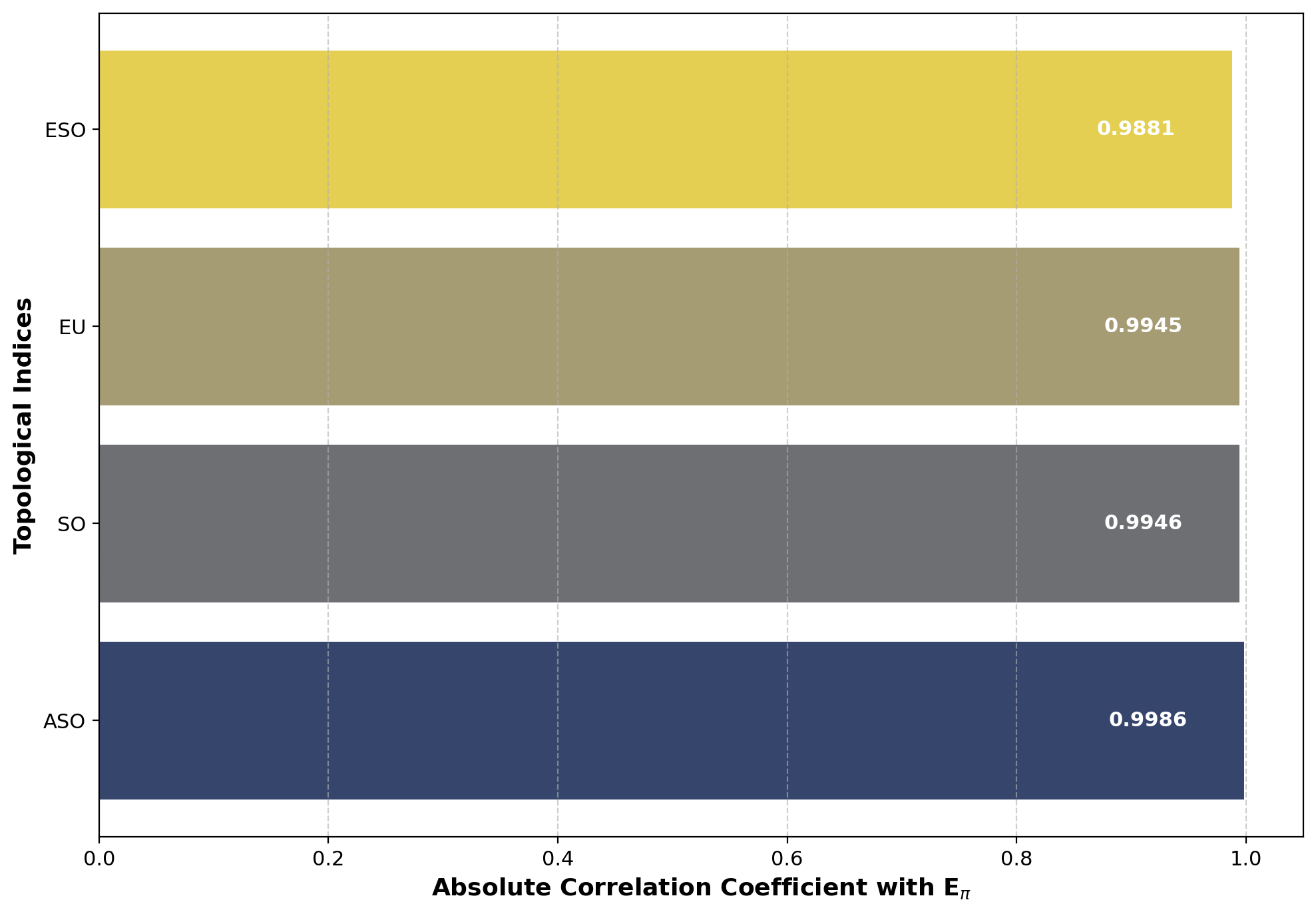}
\centering \caption{Absolute correlation plot: $E_\pi$ versus selected topological indices.}\label{fig:HVAP33.png}
\end{figure}
\vspace*{3mm}
\begin{figure}[H]
\includegraphics[width=8cm]{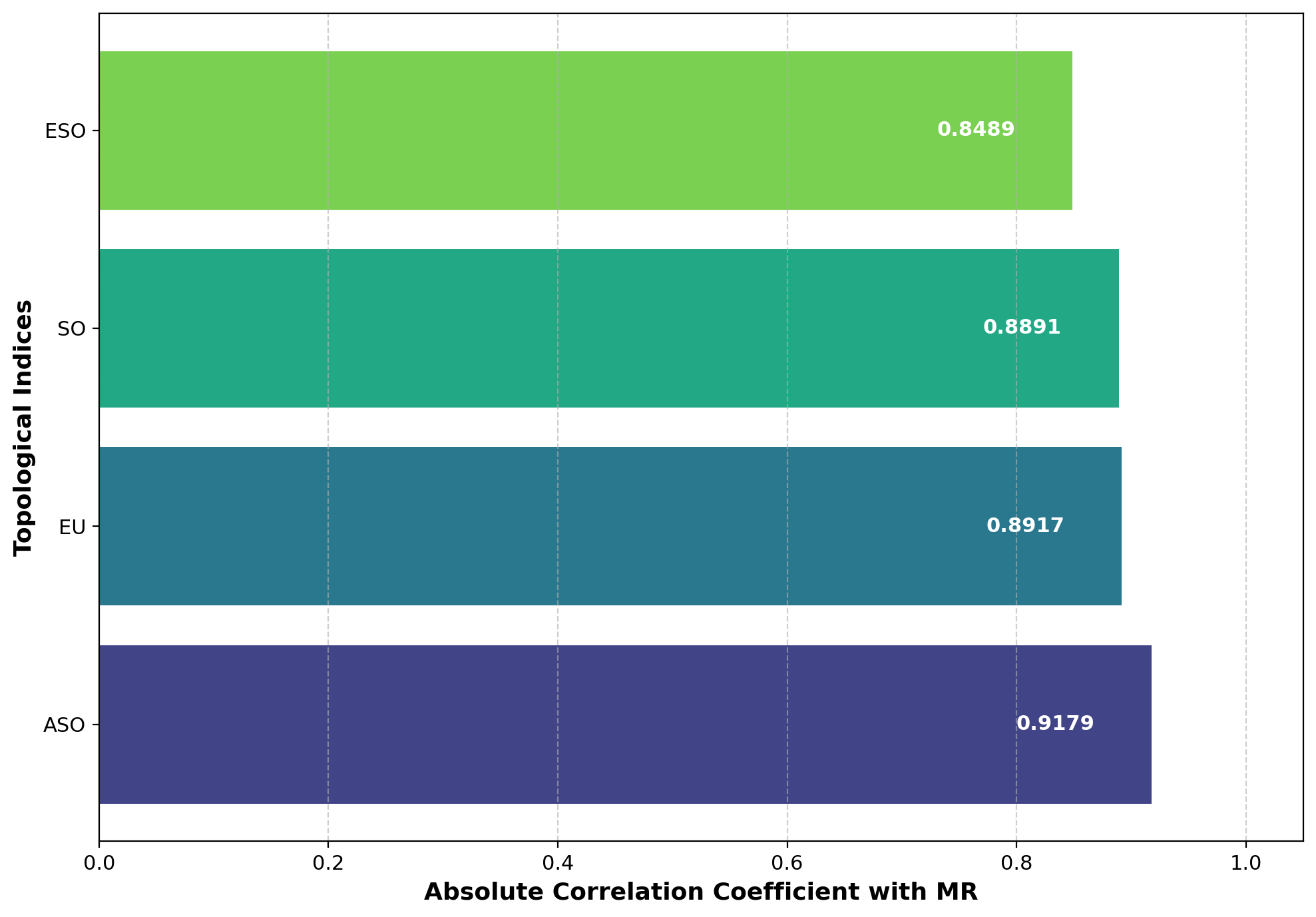}
\centering \caption{Absolute correlation plot: $MR$ versus selected topological indices.} \label{fig:DHVAP44.png}
\end{figure}
\vspace*{3mm}
\begin{figure}[H]
\includegraphics[width=8cm]{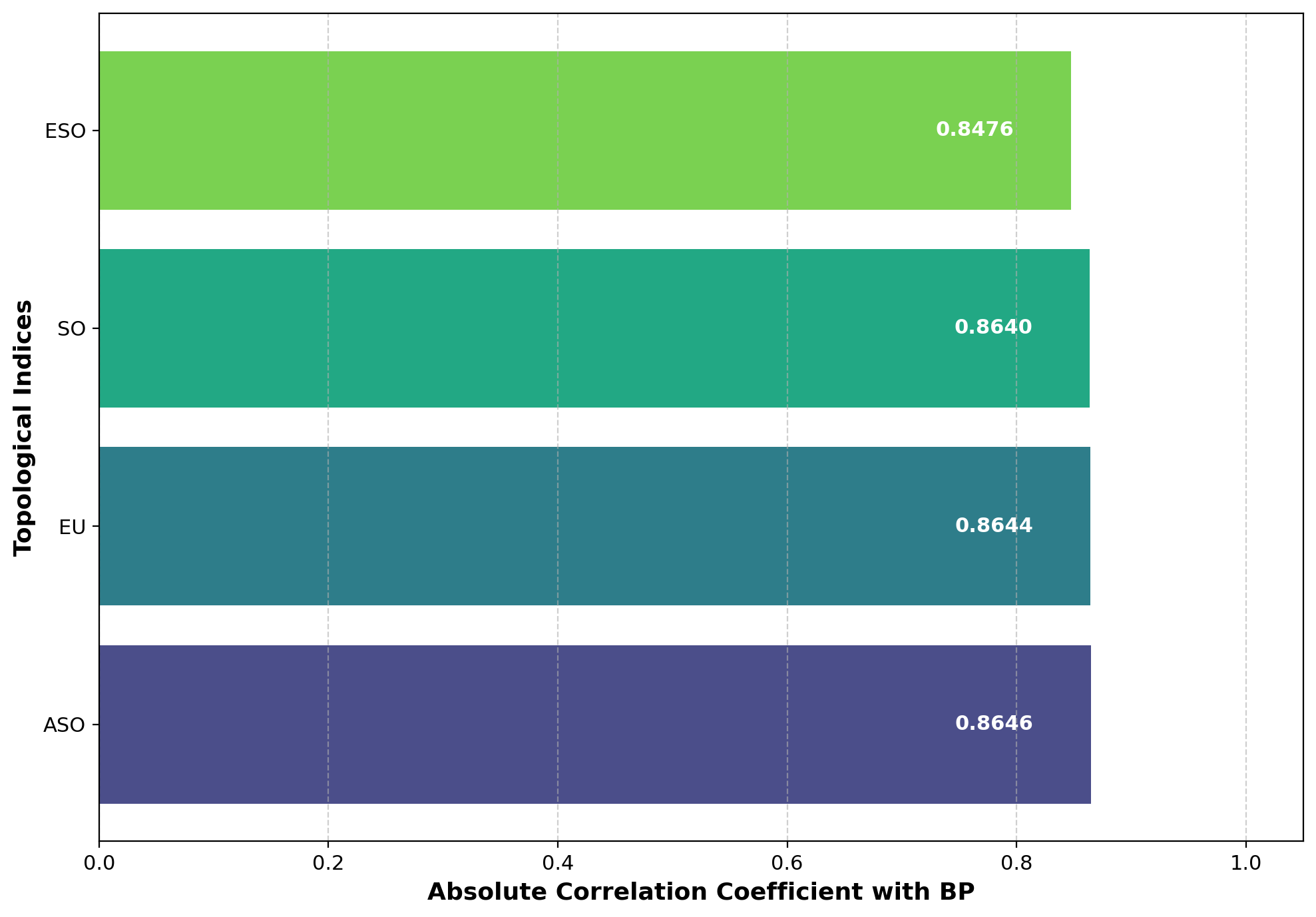}
\centering \caption{Absolute correlation plot: $BP$ versus selected topological indices.}\label{fig:HVAP44.png}
\end{figure}

\section{Concluding Remarks and Future Work}\label{section5}

In this paper, we investigated the augmented Sombor index ($ASO$) from both theoretical and chemical perspectives. Several sharp bounds for the $ASO$ index were established over important graph classes, and the corresponding extremal structures were completely characterized. In particular, we identified the minimum value of the $ASO$ index for unicyclic graphs with a given girth and determined the unique graphs attaining this bound. Moreover, the second maximum $ASO$ value among trees was obtained, together with a characterization of the extremal tree. We also derived the minimum and maximum values of $ASO$ for bipartite graphs and chemical graphs and identified all extremal graphs achieving these values. In addition, maximal chemical trees with respect to the $ASO$ index were characterized.
From the chemical viewpoint, the applicability of the $ASO$ index was examined via quantitative structure–property relationship (QSPR) analysis. The results indicate that the $ASO$ index shows strong predictive ability for relevant physicochemical properties, confirming its effectiveness as a molecular descriptor.

Motivated by the extremal results obtained in this work, we further propose a conjecture concerning the maximum value of the $ASO$ index over multipartite graphs. Recall that the Turán graph $T_n(k)$ is the complete $k$-partite graph of order $n$ whose partition sets differ in size by at most one. We conjecture that $T_n(k)$ attains the maximum $ASO$ value among all $k$-partite graphs of order $n$. If this conjecture holds, then the conclusion of Theorem~\ref{theorem5.2} follows as a special case.

\begin{conjecture}
For any $k$-partite graph $\Omega$ of order $n$,
\begin{align}
ASO(\Omega)&\leq r(k-r)\Big\lceil\frac{n}{k}\Big\rceil\Big\lfloor\frac{n}{k}\Big\rfloor\sqrt{\frac{\left(n-\lceil\frac{n}{k}\rceil\right)^2+\left(n-\lfloor\frac{n}{k}\rfloor\right)^2}{n-2}}+\binom{r}{2}\Big\lceil\frac{n}{k}\Big\rceil^2\frac{n-\lceil\frac{n}{k}\rceil}{\sqrt{n-\lceil\frac{n}{k}\rceil-1}}\nonumber\\[2mm]
&~~~~+\binom{k-r}{2}\Big\lfloor\frac{n}{k}\Big\rfloor^2\frac{n-\lfloor\frac{n}{k}\rfloor}{n-\lfloor\frac{n}{k}\rfloor-1}\nonumber
\end{align}
with equality if and only if $G\cong T_n(k)$.
\end{conjecture}
Future research may extend the study of the augmented Sombor index to other graph classes and explore further extremal problems and chemical applications.

\end{document}